\documentclass[12pt]{article}

\usepackage[T1]{fontenc}
\usepackage{lmodern}
\usepackage{amsmath,amssymb,amsthm,mathtools}
\usepackage[margin=1in]{geometry}
\usepackage{microtype}
\usepackage{needspace}
\usepackage[hidelinks]{hyperref}
\hypersetup{
  pdftitle={Stability of isometries of quantum states},
  pdfauthor={Sam Looi}
}

\usepackage{setspace}
\numberwithin{equation}{section}

\newtheorem{theorem}{Theorem}[section]
\newtheorem{lemma}[theorem]{Lemma}

\newtheorem{corollary}[theorem]{Corollary}
\theoremstyle{definition}
\newtheorem*{stabilityproblem}{Hyers--Ulam stability problem}
\theoremstyle{remark}
\newtheorem{remark}[theorem]{Remark}
\newtheorem*{unnumberedremark}{Remark}

\newcommand{\C}{\mathcal C}
\newcommand{\St}{\mathcal S}
\newcommand{\Tr}{\operatorname{Tr}}
\newcommand{\Fid}{\operatorname{F}}
\newcommand{\db}{d_{\mathrm b}}
\newcommand{\done}{d_1}

\title{Stability of isometries of quantum states}
\author{Sam Looi}
\date{October 13, 2026}

\begin{document}

\maketitle \begin{abstract} 
We prove dimension-independent Hyers--Ulam stability for surjective approximate isometries of the positive trace-class cone in both the trace and Bures metrics. Here $\varepsilon\geq0$ is a uniform upper bound on the additive error in preserving distances. Such an approximate isometry is uniformly approximated by a Wigner symmetry (a unitary or antiunitary conjugation), with error at most $9\varepsilon/2$ in the trace metric and $2\sqrt2\varepsilon$ in the Bures metric, even if it does not fix zero. For maps fixing zero, the trace bound improves to $3\varepsilon$ even under the hypothesis of $\delta$-surjectivity, independently of $\delta$. In infinite dimensions, we construct spectrum-preserving homeomorphisms of the full state space, including mixed states, whose distortions tend to zero in both metrics but which stay a fixed uniform distance from every Wigner symmetry. For state spaces, stability holds in each fixed finite dimension with a dimension-dependent modulus tending to zero as $\varepsilon\to0$, but no such modulus can be chosen independently of dimension.
\end{abstract}

\section{Introduction} Let $H$ be a nonzero complex Hilbert space. In quantum mechanics, bounded observables are represented by bounded self-adjoint operators on $H$, while normal states are represented by positive trace-class operators on $H$ with trace one. Several distances (that is, metrics) between quantum states have been studied in connection with different physical problems; Hadjisavvas \cite{Had} discusses the metrics of Jauch, Misra and Gibson, Gudder, Cantoni and Wootters. The definitions of each of these metrics may look very different, but each distance can be obtained from either the Bures or trace distance by a simple formula \cite{Had,MT}. This motivates our consideration of the Bures and trace distances in this article.

Moln\'ar and Timmermann \cite{MT} considered four isometry problems: the positive cone and the state space, each equipped with either the Bures or trace metric. The purpose of this paper is to study stability in these four cases. We show that both cone cases admit estimates which do not depend on the dimension. For the Bures metric the bound is $2\sqrt2 \varepsilon$, without assuming that zero is preserved, while for the trace metric the bound is $9\varepsilon/2$, and it improves to $3\varepsilon$ when zero is preserved. In the latter case, the same bound holds under $\delta$-surjectivity for any finite $\delta$, as defined below. For the two state space cases, stability fails in infinite dimension and no modulus tending to zero can be uniform in the dimension. In each fixed finite dimension, however, stability holds with a dimension-dependent modulus tending to zero.

\subsection{The spaces and their metrics}

Given a bounded operator $A$ on $H$, let $A^*$ denote its adjoint. We write $A\geq0$ if $A$ is self-adjoint and $\langle Ax,x\rangle\geq0$ for every $x\in H$, and $A\leq B$ if $B-A\geq0$. For $A\geq0$, $A^{1/2}$ will denote its unique positive square root.

We recall some notation. If $A\geq0$ and $(e_j)$ is an orthonormal basis, then $\Tr A=\sum_j\langle Ae_j,e_j\rangle$; this value does not depend on the basis. An operator $X$ is trace-class if $\Tr|X|<\infty$, where $|X|=(X^*X)^{1/2}$. The space of these operators is $S_1(H)$, and $S_1(H)_+$ denotes its positive cone. For $X\in S_1(H)$,
\[
 \|X\|_1=\Tr|X|.
\]
For a bounded operator $X$ on $H$, its operator norm is
\[
 \|X\|_\infty=\sup_{\|x\|=1}\|Xx\|.
\]
For a trace-class operator $X$, let $s_j(X)$ denote its singular values, that is, the eigenvalues of $|X|$ counted with multiplicity. Then
\[
 \|X\|_1=\sum_j s_j(X),
 \qquad
 \|X\|_\infty=\sup_j s_j(X).
\]
Thus the subscripts $1$ and $\infty$ correspond to the $\ell^1$ and $\ell^\infty$ norms of the singular-value sequence. We put
\[
 \C=S_1(H)_+,
 \qquad
 \St=\{A\in\C:\Tr A=1\}.
\]
Thus $\C$ is the positive trace-class cone and $\St$ is its state space. For positive trace-class operators, $\|X\|_1=\Tr X$. The set $\C$ is a cone because it is closed under addition and multiplication by nonnegative scalars. The state space $\St$ is its trace-one part.

For $A,B\in\C$, we make the following definitions
\begin{equation}
 \begin{aligned}
 \db(A,B)^2&:=\Tr A+\Tr B
 -2\Tr \left[(A^{1/2}BA^{1/2})^{1/2}\right],\\
 \done(A,B)&:=\|A-B\|_1.
 \end{aligned}
 \label{eq:bures-definition}
\end{equation}
The elements of $\St$ are the density operators, also called quantum states.

For operators diagonal in the same orthonormal basis, say $A=\operatorname{diag}(a_j)$ and $B=\operatorname{diag}(b_j)$, these formulas become
\[
 \db(A,B)^2=\sum_j(\sqrt{a_j}-\sqrt{b_j})^2,\qquad
 \done(A,B)=\sum_j|a_j-b_j|.
\]
Thus, in this case, Bures distance is the Hilbert space distance between the square root sequences, while trace distance is the $\ell^1$ distance between the original sequences.

We study the Bures and trace metrics, the two metrics in the classification of Moln\'ar and Timmermann.

\subsection{The stability problem and main results}

Recall that an isometry is a map that preserves every distance, while a unitary is a linear surjection $U:H\to H$ satisfying $\langle Ux,Uy\rangle=\langle x,y\rangle$. A Wigner symmetry is defined to be a map $A\mapsto UAU^*$ with the map $U$ either unitary or antiunitary.

One can equip either the positive trace-class cone or the state space with either the Bures metric or the trace metric, giving a total of four metric spaces. The basic question is whether, for each of the four metric spaces, every bijective isometry is a Wigner symmetry. As reported by Moln\'ar and Timmermann \cite{MT}, Uhlmann asked for the classification of Bures isometries of the state space and the positive trace-class cone, and for their relation to the symmetry transformations of quantum mechanics. Moln\'ar and Timmermann \cite[Theorems~1--4]{MT} classified the bijective isometries for both the Bures and trace metrics, on both the positive trace-class cone and the state space. Their results answer Uhlmann's Bures isometry questions and establish the corresponding classifications for the trace metric: in all four settings, the bijective isometries are exactly the Wigner symmetries. An antiunitary is a conjugate linear surjection $U:H\to H$ satisfying $\langle Ux,Uy\rangle=\overline{\langle x,y\rangle}$. Although $U$ is not linear, $UAU^*$ is a linear operator whenever $A$ is, and this conjugation preserves positivity and the trace of positive operators.

\begin{theorem}[Moln\'ar-Timmermann]\label{thm:MT} For each of the four metric spaces
\[
 (\C,\db),\qquad(\C,\done),\qquad(\St,\db),\qquad(\St,\done),
\]
every bijective isometry is of the form
\begin{equation}
 A\mapsto UAU^*,
 \label{eq:wigner-form}
\end{equation}
where $U$ is unitary or antiunitary on $H$. \end{theorem}

Throughout the paper, and in particular in our theorems, no linearity, affinity or continuity assumptions whatsoever on the approximate isometries are assumed.

We consider additive errors: let $\varepsilon\geq0$, and let $d$ be either $\db$ or $\done$. A map $S:\C\to\C$ is called an $\varepsilon$-isometry if
\begin{equation}
 \bigl|d(S(A),S(B))-d(A,B)\bigr|\leq\varepsilon,
 \qquad A,B\in\C.
 \label{eq:eps-isometry}
\end{equation}
The $\varepsilon$-isometry $S$ is called standard if $S(0)=0$. On $\St$, where zero is absent, we only require the distance inequality. A map into a metric space is $\delta$-surjective if every point of that space lies within distance $\delta$ of some value of the map (hence surjectivity is the case $\delta=0$). We use the same terminology for maps between Banach spaces with their norm metrics.

We study quantitative versions of Theorem~\ref{thm:MT}. \begin{stabilityproblem} Fix one of the four metric spaces $(X,d)$ above. We ask whether there is a constant $K<\infty$ such that, for every $\varepsilon\geq0$ and every surjective $\varepsilon$-isometry $T:X\to X$, there is a unitary or antiunitary $U$ on $H$ satisfying
\[
 \sup_{A\in X}d\bigl(T(A),UAU^*\bigr)\leq K\varepsilon.
\]
The constant $K$ may depend on the metric space, but must be independent of $\varepsilon$ and $T$; the same unitary or antiunitary $U$ must satisfy this error bound for every $A\in X$. \end{stabilityproblem}

Our cone results give such bounds with constants independent of dimension, without requiring the maps to fix zero.\footnote{The operator $U$ acts on the Hilbert space $H$, whereas the corresponding isometry on operators is $\Phi(A)=UAU^*$. For the cone equipped with the trace metric, $\Phi$ extends to a real-linear surjective isometry of the real Banach space $E=S_1(H)_{\mathrm{sa}}$ (``sa'' for self-adjoint), equipped with the trace norm. The estimate in that case would be $\|T(A)-\Phi(A)\|_1\leq K\varepsilon$ for every $A\in E_+$. Thus $\Phi$ plays the role of the linear approximating isometry in Banach space stability. The domain of $T$, however, is only the positive cone $E_+$, so this is a positive cone version of that problem.}

For the state space $\mathcal{S}$, we also consider the weaker stability question obtained by replacing $K\varepsilon$ with a function $\omega(\varepsilon)\to0$ as $\varepsilon\to0$, again independent of the particular map. Even this weaker stability fails in infinite dimension. It holds in each fixed finite dimension, but no such function can be chosen independently of dimension. Thus our answer to the Hyers--Ulam stability problem on the cone is uniform in the dimension and the constants are explicit, while on the state space the dependence on the dimension is unavoidable, for both metrics.

Our cone arguments use zero as a reference point and in particular, the trace can be recovered from the Bures distance to zero: $\Tr A=\db(A,0)^2.$ These cone arguments also use arbitrarily large scalar multiples. Neither zero nor these large scalar multiples belong to the state space, whose elements all have trace one (Section~\ref{sec:spheres}).

We first prove the result for the Bures metric on the cone.

\begin{theorem}\label{thm:bures-main} Let $S:(\C,\db)\to(\C,\db)$ be a surjective $\varepsilon$-isometry. There is a unitary or antiunitary $U$ on $H$ such that
\begin{equation}
 \db\bigl(S(A),UAU^*\bigr)\leq2\sqrt2 \varepsilon,
 \qquad A\in\C.
 \label{eq:bures-main}
\end{equation}
Moreover,
\begin{equation}
 \db(S(0),0)\leq\frac{1+\sqrt7}{2}\varepsilon.
 \label{eq:bures-zero}
\end{equation}
If $S(X)=0$, then $\db(X,0)\leq\frac{1+\sqrt7}{2}\varepsilon$ as well. Neither $S(0)=0$ nor injectivity is assumed. \end{theorem}

When $S(0)=0$, the bound~\eqref{eq:bures-main} remains valid under $\delta$-surjectivity, independently of $\delta$; see Corollary~\ref{cor:bures-delta-surjective}.

After that, we prove the corresponding result for the trace metric.

\begin{theorem}\label{thm:trace-global} Let $T:(\C,\done)\to(\C,\done)$ be a surjective $\varepsilon$-isometry. There is a unitary or antiunitary $U$ on $H$ such that
\begin{equation}
 \|T(A)-UAU^*\|_1\leq\frac92\varepsilon,
 \qquad A\in\C.
 \label{eq:trace-global}
\end{equation}
If $T(0)=0$, the right side can be replaced by $3\varepsilon$. The estimate \eqref{eq:trace-global} requires neither injectivity nor $T(0)=0$. In addition, every $X\in\C$ with $T(X)=0$ satisfies
\begin{equation} \Tr T(0)+\Tr X\leq3\varepsilon. \label{eq} \end{equation}
\end{theorem}

The $3\varepsilon$ bound remains valid when surjectivity is weakened to $\delta$-surjectivity, with no dependence on $\delta$:

\begin{corollary}\label{cor:trace-delta-surjective} Let $\varepsilon,\delta\geq0$, and let $T:(\C,\done)\to(\C,\done)$ be a $\delta$-surjective $\varepsilon$-isometry with $T(0)=0$. There is a unitary or antiunitary $U$ on $H$ such that
\begin{equation}
 \|T(A)-UAU^*\|_1\leq3\varepsilon,
 \qquad A\in\C.
 \label{eq:trace-delta-surjective}
\end{equation}
\end{corollary}

The surjectivity hypotheses in the cone results, including $\delta$-surjectivity in the corollary, cannot simply be omitted, even for exact isometries; see Remark 2.1.

We next state the results for the state space, whose proofs are given in Section \ref{sec:spheres}.

\begin{theorem}\label{thm:sphere-instability} Let $H$ be infinite dimensional. There are spectrum-preserving homeomorphisms $\Theta_n:\St\to\St$ for both metrics and numbers $\varepsilon_n\to0$ such that each $\Theta_n$ is an $\varepsilon_n$-isometry for both $\db$ and $\done$, but
\begin{align}
 \inf_U\sup_{\rho\in\St}
 \db(\Theta_n(\rho),U\rho U^*)
 &\geq\beta_{\mathrm b}
 :=\sqrt{2-\sqrt 3},
 \label{eq:fixed-bures-gap}\\
 \inf_U\sup_{\rho\in\St}
 \|\Theta_n(\rho)-U\rho U^*\|_1
 &\geq 1,
 \label{eq:fixed-trace-gap}
\end{align}
where the infima run over all unitary and antiunitary $U$. As a result, for neither metric is there a function $\omega(\varepsilon)\to0$ such that every surjective $\varepsilon$-isometry is within $\omega(\varepsilon)$ of a unitary or antiunitary conjugation. No such function can be uniform in the dimension. \end{theorem}

\begin{corollary}\label{cor:fixed-dimensional} For each $d<\infty$ and each choice of $\mathsf d\in\{\db,\done\}$ on $\St(\mathbb C^d)$, there is a function $\omega_{d,\mathsf d}(\varepsilon)\to0$ as $\varepsilon\to0^+$ such that every $\varepsilon$-isometry $T:(\St(\mathbb C^d),\mathsf d)\to(\St(\mathbb C^d),\mathsf d)$ satisfies
\[
 \sup_{\rho\in\St(\mathbb C^d)}
 \mathsf d\bigl(T(\rho),U\rho U^*\bigr)
 \leq\omega_{d,\mathsf d}(\varepsilon)
\]
for some unitary or antiunitary $U$. \end{corollary}

No surjectivity assumption on $T$ is needed.

\subsection{Related work}

The stability problem for isometries was introduced by Hyers and Ulam \cite{HU}. For real Banach spaces, Gevirtz \cite{Gevirtz} proved that every standard surjective $\varepsilon$-isometry is uniformly within $5\varepsilon$ of a surjective linear isometry. Omladi\v{c} and \v{S}emrl \cite{OS} obtained the optimal bound $2\varepsilon$. \v{S}emrl and V\"ais\"al\"a \cite{SV} proved that the same bound holds for standard $\delta$-surjective $\varepsilon$-isometries, independently of $\delta$. Vestfrid \cite{Vestfrid15} obtained a $2\varepsilon$ approximation by an affine surjective isometry under a weaker asymptotic condition on the range.

Weak stability for nonsurjective maps was studied by Cheng, Dong and Zhang \cite{CDZ} and by Cheng and Dong \cite{CD}. In the form proved in \cite{CD}, if $f:X\to Y$ is a standard $\varepsilon$-isometry between real Banach spaces, then for every $x^*\in X^*$ there is $\varphi\in Y^*$ with $\|\varphi\|=\|x^*\|$ such that
\begin{equation}
 |\varphi(f(x))-x^*(x)|\leq3\varepsilon\|x^*\|,
 \qquad x\in X.
 \label{eq:weak-stability-background}
\end{equation}
The functional $\varphi$ may depend on $x^*$. In our trace norm argument, positive operators with complementary supports identify the limit of the norming functionals with a prescribed sign operator.

The preceding results concern maps defined on whole Banach spaces. Cheng, Tu and Zhang \cite{CTZ} studied weak stability for maps defined on wedges of Banach spaces. For positive cones of $L^p$ spaces, $1<p<\infty$, Sun \cite{Sun} proved a stability theorem with bound $4\varepsilon$ for standard almost-surjective $\varepsilon$-isometries. Sun et al. \cite{SSW} obtained the bound $2\varepsilon$ for standard $\delta$-surjective maps between positive cones of continuous function spaces when the underlying compact Hausdorff spaces are perfectly normal. Vestfrid \cite{Vestfrid23} proved a sharp $2\varepsilon$ bound for standard surjective maps between such cones over arbitrary compact Hausdorff spaces.

Dong, Leung and Li \cite{DLL} proved stability results for several classes of ordered Banach spaces, including the self-adjoint Schatten classes for $1<p<\infty$ with the Schatten class result not covering the trace-class endpoint $p=1$. For positive $L^1$ cones, the same three authors and Zhang \cite{DLLZ} then proved the optimal bound $2\varepsilon$ for standard $\delta$-surjective maps. Their Theorem~3.15 gives a uniform operator norm estimate: for every standard surjective $\varepsilon$-isometry $R$ of $\C$, there is a unitary or antiunitary $U$ on $H$ such that
\[
 \sup_{A\in\C}\|R(A)-UAU^*\|_\infty<\infty.
\]
Their Corollary~3.16 gives a trace norm estimate in finite dimension with a constant proportional to the dimension. Passing from operator norm to trace norm using $\|X\|_1\leq d\|X\|_\infty$ costs a factor of $d$ and gives no bound in infinite dimension. In the treatment in this paper we instead use the dual formula for the trace norm to estimate the error with a constant independent of dimension. Theorem~\ref{thm:trace-global} answers their question about infinite-dimensional trace norm stability, giving $3\varepsilon$ when zero is fixed and $9\varepsilon/2$ without this assumption.

\subsection{Organization of the paper} Section~\ref{sec:preliminaries} collects the operator and metric identities used below and explains the measurement interpretation of the trace metric. We then prove stability for the Bures metric on the cone (Section~\ref{sec:blowdown}). We construct an exact isometry as a limit of rescaled maps, then use distance comparisons along rays to bound its distance from the original map uniformly. After proving Bures stability, we turn to the trace metric. The main step in proving our trace norm stability results is Theorem~\ref{thm:trace-rigidity}, which upgrades a finite uniform operator norm bound to a trace norm estimate for differences, with a constant independent of dimension (Subsection~\ref{sec:trace-rigidity}). Combined with the known operator norm estimate, this gives our trace norm stability results: Theorem~\ref{thm:trace-global} for surjective maps and Corollary~\ref{cor:trace-delta-surjective} for maps that fix zero and are $\delta$-surjective. These applications are proved in Subsection~\ref{sec:trace-main}. Section~\ref{sec:spheres} constructs state space examples showing that stability fails in infinite dimension and cannot hold uniformly in the dimension. We also present a compactness argument establishing a stability modulus tending to zero in each fixed finite dimension.

\section{Preliminaries}\label{sec:preliminaries}

We collect the elementary operator facts and metric calculations used in the proofs.

\subsection{Operator notation and inequalities}

An operator $X$ is Hilbert-Schmidt if $\Tr X^*X<\infty$; its Hilbert-Schmidt norm is $\|X\|_2=(\Tr X^*X)^{1/2}$. The subscripts $1$, $2$ and $\infty$ follow the Schatten norm convention: for compact operators, these norms are the $\ell^1$, $\ell^2$ and $\ell^\infty$ norms of the singular values, respectively. We will use the Schatten form of the Cauchy-Schwarz inequality
\begin{equation}
 \|XY\|_1\leq\|X\|_2\|Y\|_2
 \label{eq:schatten-holder}
\end{equation}
and trace norm duality
\begin{equation}
 \|X\|_1=\sup_{\|K\|_\infty\leq1}|\Tr(KX)|.
 \label{eq:trace-duality}
\end{equation}
For completeness, if $XY=V|XY|$ is the polar decomposition, then $V$ acts isometrically on the closure of the range of $|XY|$ and is zero on its orthogonal complement. Such an operator is called a partial isometry. We have
\[
 \|XY\|_1=\Tr(V^*XY)
 \leq\|X^*V\|_2\|Y\|_2
 \leq\|X\|_2\|Y\|_2.
\]
This proves \eqref{eq:schatten-holder} from the ordinary Cauchy-Schwarz inequality for the Hilbert-Schmidt inner product. Formula \eqref{eq:trace-duality} says that the bounded operators form the dual space of the trace-class.

For $A\geq0$, the support projection $s(A)$ is the orthogonal projection onto the closure of the range of $A$, or equivalently onto $(\ker A)^\perp$. We say that $A$ is supported on a closed subspace $M$ if $A$ vanishes on $M^\perp$ and maps $M$ into itself.

For a self-adjoint operator $X$, its positive and negative parts are $X_+=(|X|+X)/2$ and $X_-=(|X|-X)/2$; both of these operators are positive and have orthogonal supports. The positive and negative spectral subspaces of $X$ are the ranges of its spectral projections for $(0,\infty)$ and $(-\infty,0)$, respectively.

\subsection{Metric identities and examples}

On positive definite matrices, $\db$ is also called the Bures-Wasserstein distance; see Bhatia, Jain and Lim \cite{BJL} for a matrix analysis treatment.

For $A,B\in\C$, put
\[
 \Fid(A,B):=\Tr\bigl(A^{1/2}BA^{1/2}\bigr)^{1/2}.
\]
The quantity $\Fid$ is called the root fidelity, while $\Fid(A,B)^2$ is called the fidelity. For the diagonal operators in the introduction,
\[
 \Fid(A,B)=\sum_j\sqrt{a_jb_j}.
\]

We will also use the identity
\begin{equation}
 \Fid(A,B)=\|A^{1/2}B^{1/2}\|_1.
 \label{eq:fidelity-trace-norm}
\end{equation}
Indeed, the absolute value of $B^{1/2}A^{1/2}$ is $(A^{1/2}BA^{1/2})^{1/2}$, and an operator and its adjoint have the same trace norm. We will use the identities
\begin{equation}
 \db(tA,tB)=\sqrt t\,\db(A,B),
 \qquad
 \done(tA,tB)=t\,\done(A,B),
 \qquad t\geq0.
 \label{eq:homogeneity}
\end{equation}
These identities follow from $\Fid(tA,tB)=t\Fid(A,B)$ and homogeneity of the trace norm. On the state space, the two metrics satisfy
\begin{equation}
 \db(\rho,\sigma)^2\leq\done(\rho,\sigma)
 \leq2\db(\rho,\sigma),
 \qquad \rho,\sigma\in\St.
 \label{eq:state-metric-comparison}
\end{equation}
In finite dimension, this follows from the Fuchs-van de Graaf inequalities \cite{FvdG}. Lemma \ref{lem:bures-complete} gives the corresponding estimate for arbitrary positive trace-class operators, without a dimension assumption.

We briefly recall how these objects occur in quantum theory (but we will not use the quantum formalism in the proofs). A unit vector $x\in H$ gives the rank one state $P_x=x\otimes x$, where $P_xz=\langle z,x\rangle x$. Convex combinations of the $P_x$ give examples of mixed states. A two-outcome measurement is described by $E$ and $I-E$, where $0\leq E\leq I$. In state $\rho$, the first outcome has probability $\Tr(\rho E)$, so $\Tr((\rho-\sigma)E)$ is the difference between its probabilities in the states $\rho$ and $\sigma$.

The trace metric then has the following interpretation: for two states $\rho,\sigma$, one has
\begin{equation}
 \sup_{0\leq E\leq I}
 \big|\Tr\big((\rho-\sigma)E\big)\big|
 =\frac12\done(\rho,\sigma).
 \label{eq:trace-distinguishability}
\end{equation}
Thus $\done(\rho,\sigma)/2$ is the largest possible difference between the probabilities of the same measurement outcome in the states $\rho$ and $\sigma$.

The proof of \eqref{eq:trace-distinguishability} is as follows: let $X=\rho-\sigma$, and $X=X_+-X_-$, where $X_\pm=(|X|\pm X)/2$. Since $\Tr X=0$, we have $\Tr X_+=\Tr X_-=\|X\|_1/2$. For $0\leq E\leq I$,
\[
 -\Tr X_-\leq\Tr(XE)\leq\Tr X_+.
\]
Let $E$ be the support projection of $X_+$ so $X_+E=X_+$; the positive and negative parts of $X$ act on orthogonal subspaces, so $X_-E=0$, and
\[ \Tr(XE)=\Tr X_+=\frac12\|X\|_1. \]
This attains the upper bound and proves \eqref{eq:trace-distinguishability}.

For pure states, the Bures distance is determined by the overlap of their vectors. If $x,y$ are unit vectors, then
\[ \Fid(P_x,P_y)=|\langle x,y\rangle|, \qquad \db(P_x,P_y)^2=2-2|\langle x,y\rangle|. \]
The absolute value reflects the fact that multiplying either vector by a phase does not change its state. The formula in \eqref{eq:bures-definition} extends this comparison to mixed states.

We briefly give an interpretation of the cone $\C$. If $A\in\C$ and $\Tr A\leq1$, then $A$ is often called a subnormalized state; its trace may give the probability that a conditioning step succeeds. More generally, every nonzero $A\in\C$ is a positive multiple of a state:
\[ A=(\Tr A)\rho,\qquad \rho=\frac{A}{\Tr A}\in\St. \]
This matters below because we can rescale elements of $\C$, while every element of $\St$ has trace one.

\begin{remark} \label{rem:range-assumption} The cone conclusions fail without an assumption on the range, even when $\varepsilon=0$. For example, suppose that $H$ is infinite dimensional, let $V:H\to H$ be an isometry whose range is not all of $H$, and choose a nonzero $P\in\C$ supported on $(VH)^\perp$. The map
\[
 A\mapsto VAV^*+P
\]
is an exact isometric embedding for both metrics. For the trace metric this is immediate, since the fixed term $P$ cancels and the trace norm is preserved under conjugation by an isometry. For the Bures metric, the orthogonal block decomposition gives
\[
 \Fid(VAV^*\oplus P,VBV^*\oplus P)=\Fid(A,B)+\Tr P,
\]
so the terms involving $P$ cancel in \eqref{eq:bures-definition}. The map is not onto, and it sends zero to $P$, so it is not one of the conjugations in \eqref{eq:wigner-form}.

Taking $P=0$ instead gives a standard exact isometric embedding $A\mapsto VAV^*$ which is not surjective. Thus the range assumption cannot be omitted even for maps fixing zero. Indeed, this embedding is not $\delta$-surjective for any finite $\delta$: positive multiples of a rank one projection supported on $(VH)^\perp$ have unbounded distance from its range in both metrics. \end{remark}

For separable $H$, Busch \cite{Busch} classified the linear trace norm isometries of the positive cone without assuming surjectivity.

\section{Bures stability on the cone}\label{sec:blowdown}

This section proves Theorem~\ref{thm:bures-main}, including its estimates concerning zero. Throughout the section, $S:(\C,\db)\to(\C,\db)$ is an $\varepsilon$-isometry; surjectivity is assumed only where stated. The proof has two main steps. First, we construct an exact isometry by rescaling. For $t>0$, the identity $\db(t^2A,t^2B)=t\,\db(A,B)$ shows that the map $S_t(A)=t^{-2}S(t^2A)$ is an $(\varepsilon/t)$-isometry. Lemma~\ref{lem:blowdown} proves that these maps converge pointwise in Bures distance along $t=2^n\to\infty$ to an exact isometry $\Phi$ fixing zero. If $S$ is surjective, Lemma~\ref{lem:phi-onto} shows that $\Phi$ is also surjective, and Theorem~\ref{thm:MT} gives $\Phi(A)=UAU^*$. Second, we estimate the distance between $S$ and $\Phi$. Distance comparisons along rays, together with the Busemann limit formula, give $|\Fid(B,S(A))-\Fid(B,\Phi(A))|\leq2\varepsilon$ for every $A\in\C$ and every state $B$. The fidelity profile estimate in Lemma~\ref{lem:profile} then yields \eqref{eq:bures-main}. Further use of the same distance comparisons gives the estimates concerning zero.

\subsection{Construction of an exact isometry} For $A\in\C$, we put
\[
 r(A)=\sqrt{\Tr A}=\db(A,0).
\]
An amplitude of $A\in\C$ is a Hilbert-Schmidt operator $X$ such that $XX^*=A$. In particular, $A^{1/2}$ is an amplitude of $A$, and every amplitude satisfies
\[
 \|X\|_2^2=\Tr(XX^*)=\Tr A.
\]
Fidelity is hard to estimate directly; Uhlmann's theorem expresses it through amplitudes and turns the Bures distance into a Hilbert-Schmidt distance between amplitudes, where the Cauchy-Schwarz inequality is available. Our usage of this theorem is why amplitudes appear. We shall use the following form of the theorem \cite{Uhlmann}. If $X$ is any fixed amplitude of $A$, then
\begin{align}
 \Fid(A,B)
 &=\sup_{YY^*=B}\operatorname{Re}\Tr(X^*Y),
 \label{eq:amplitude-fidelity}\\
 \db(A,B)
 &=\inf_{YY^*=B}\|X-Y\|_2.
 \label{eq:amplitude-distance}
\end{align}
The supremum and infimum are taken over all amplitudes $Y$ of $B$. The second formula follows from the first because
\[
 \|X-Y\|_2^2
 =\Tr A+\Tr B-2\operatorname{Re}\Tr(X^*Y).
\]
Thus the Bures distance is the infimum of the Hilbert-Schmidt distances between such amplitudes. Moreover, if $X,Y$ are amplitudes of $A,B$, then
\[
 \|X^*Y\|_1=\Fid(A,B).
\]
Indeed, $\|X^*Y\|_1=\|Y^*X\|_1$, and the squared singular values of $Y^*X$ are the eigenvalues of $X^*BX$. These are the nonzero eigenvalues of $B^{1/2}AB^{1/2}$, which are also those of $A^{1/2}BA^{1/2}$. This gives the displayed identity and also shows that fidelity is symmetric. If $X^*Y\geq0$, its trace equals its trace norm, the pair $X,Y$ attains the supremum in \eqref{eq:amplitude-fidelity}, and $\Fid(A,B)=\Tr(X^*Y)$. This is useful because it allows one to compute the fidelity by an ordinary trace. These amplitude formulas hold for both separable and non-separable Hilbert spaces $H$, because for trace-class operators, their ranges lie in a separable closed subspace anyway.

We will also use the following square root estimate for positive trace-class operators, which was proved in the work \cite{PS}:
\begin{equation}
 \|A^{1/2}-B^{1/2}\|_2^2\leq\|A-B\|_1.
 \label{eq:powers-stormer}
\end{equation}

We first prove a comparability result which will also give completeness of the metric space $(\C,\db)$.

\begin{lemma}\label{lem:bures-complete} For $A,B\in\C := S_1(H)_+$,
\begin{equation}
 \db(A,B)^2\leq\|A-B\|_1
 \leq \db(A,B)\bigl(r(A)+r(B)\bigr).
 \label{eq:bures-trace-comparison}
\end{equation}
Thus $(\C,\db)$ is a complete metric space. \end{lemma}

\begin{proof} We begin with a proof of the first inequality. By \eqref{eq:amplitude-distance}, the particular amplitudes $A^{1/2}$ and $B^{1/2}$ give
\[
 \db(A,B)^2\leq\|A^{1/2}-B^{1/2}\|_2^2.
\]
We combine this with \eqref{eq:powers-stormer} to obtain the first inequality.

We prove the second inequality. Let us choose amplitudes $X,Y$ with $XX^*=A$, $YY^*=B$ and $\|X-Y\|_2$ arbitrarily close to $\db(A,B)$. Since
\[
 A-B=(X-Y)X^*+Y(X^*-Y^*),
\]
\eqref{eq:schatten-holder} gives
\[
 \|A-B\|_1
 \leq\|X-Y\|_2\|X\|_2+\|Y\|_2\|X-Y\|_2.
\]
We have $\|X\|_2=r(A)$ and $\|Y\|_2=r(B)$. We now let $\|X-Y\|_2$ tend to $\db(A,B)$. The second inequality follows.

It remains to prove completeness. If $(A_n)$ is Bures Cauchy, then $r(A_n)=\db(A_n,0)$ is Cauchy by the reverse triangle inequality. In particular, the sequence $(r(A_n))$ is bounded. The second inequality now shows that $(A_n)$ is trace norm Cauchy; let $A$ be its trace norm limit. Trace norm convergence implies operator norm convergence; by the operator norm convergence, $\langle Ax,x\rangle=\lim_n\langle A_nx,x\rangle\geq0$ for every $x\in H$, so $A\geq0$. The first inequality gives $\db(A_n,A)^2\leq\|A_n-A\|_1\to0$. \end{proof}

For $t>0$, we define
\begin{equation}
 S_t(A):=t^{-2}S(t^2A).
 \label{eq:blowdown}
\end{equation}
The family $S_t$ is the blow-down of $S$: it examines $S$ on larger and larger scales and rescales the result back to the original scale and, as $t\to\infty$, it captures the large-scale, or asymptotic, behavior of $S$. In our present setting, the $t^2$ (rather than a $t$) is natural, because $\db(t^2A,t^2B)=t\db(A,B)$, that is, $t^2$-scaling of operators corresponds to $t$-scaling of Bures distances.

The map $S_t$ is an $(\varepsilon/t)$-isometry with respect to the Bures metric. Indeed, using the homogeneity $ d_{\mathrm b}(t^2A,t^2B)=t\,d_{\mathrm b}(A,B)$, we have
\[ d_{\mathrm b}(S_t(A),S_t(B)) =\frac1t\,d_{\mathrm b}\bigl(S(t^2A),S(t^2B)\bigr). \]
Since $S$ is an $\varepsilon$-isometry,
\[ \left| d_{\mathrm b}\bigl(S(t^2A),S(t^2B)\bigr) -d_{\mathrm b}(t^2A,t^2B) \right|\le \varepsilon. \]
Dividing by $t$ and using the earlier display gives
\[
\left| d_{\mathrm b}(S_t(A),S_t(B)) -d_{\mathrm b}(A,B) \right| \le \frac{\varepsilon}{t}. \]
We will obtain an exact isometry by letting $t$ tend to infinity through powers of two.

\begin{lemma}[Blow-down]\label{lem:blowdown} Let $\C := S_1(H)_+$. Let $S:(\C,\db)\to(\C,\db)$ be an $\varepsilon$-isometry, and let $S_t(A):=t^{-2}S(t^2A)$ for $t>0$. For every $A\in\C$, the limit
\begin{equation}
 \Phi(A):=\lim_{n\to\infty}S_{2^n}(A)
 \label{eq:phi-limit}
\end{equation}
exists in Bures distance and belongs to $\C$. The resulting map $\Phi:\C\to\C$ is an exact isometry satisfying $\Phi(0)=0$. \end{lemma}

\begin{proof} Recall that $S : (S_1(H)_+,\db) \to (S_1(H)_+,\db)$. We put $\alpha=r(S(0))$, and $\delta=\varepsilon+\alpha$. We fix $A$ and set $m=r(A)$. For $0<t<s$, we use the following notation:
\[
 X=S(t^2A),\quad Y=S(s^2A),\quad r=r(X),\quad q=r(Y),
 \quad D=\db(X,Y),\quad L=(s-t)m.
\]
The approximate isometry inequality gives $|\db(X,S(0))-tm|\leq\varepsilon$ and by the reverse triangle inequality one has $|r-tm|\leq\delta$.

By the Bures formula, one has
\[
 \db(t^2A,s^2A)=(s-t)\sqrt{\Tr A}=L
\]
so that $L$ is the distance between the inputs, while $D$ is the distance between their images under $S$. By the approximate isometry inequality,
\[
 |D-L|\leq\varepsilon\leq\delta,
\]
where the last inequality follows from $\delta=\varepsilon+\alpha$ and $\alpha\geq0$.

Thus
\begin{equation}
 |r-tm|\leq\delta,\qquad |q-sm|\leq\delta,\qquad
 |D-L|\leq\delta.
 \label{eq:radial-errors}
\end{equation}
Since $\Tr X=r^2$ and $\Tr Y=q^2$, the Bures distance formula \eqref{eq:bures-definition} gives
\[
 D^2=r^2+q^2-2f, \qquad f:=\Fid(X,Y).
\]
Applying that same Bures distance formula to $t^{-2}X$ and $s^{-2}Y$, and using $2(rq-f)=D^2-(r-q)^2$, we obtain
\begin{equation}
 \db(S_tA,S_sA)^2
 =\left(\frac rt-\frac qs\right)^2
   +\frac{D^2-(r-q)^2}{ts}.
 \label{eq:cone-algebra}
\end{equation}
The first term compares the rescaled distances from zero. For the second term, note that points on the same ray satisfy $D=|r-q|$, so its numerator vanishes. Since $t^2A$ and $s^2A$ lie on the same ray, \eqref{eq:radial-errors} lets us bound the deviation $D-|r-q|$ for their images $X$ and $Y$. The numerator $D^2-(r-q)^2$ is nonnegative by the reverse triangle inequality:
\[
 |r-q|=|\db(X,0)-\db(Y,0)|\leq\db(X,Y)=D.
\]

By \eqref{eq:radial-errors} and $|r-q|\leq D$,
\[
 0\leq D-|r-q|\leq3\delta,
 \qquad D+|r-q|\leq2D\leq2(L+\delta).
\]
``Multiplying'' these two inequalities gives
\begin{equation}
 0\leq D^2-(r-q)^2\leq6\delta(L+\delta).
 \label{eq:angular-excess}
\end{equation}

With $s:=2t$, the first two estimates in \eqref{eq:radial-errors} give
\[
 \left|\frac rt-\frac q{2t}\right|
 \leq\frac{3\delta}{2t},
\]
and now \eqref{eq:cone-algebra} and \eqref{eq:angular-excess} yield
\begin{equation}
 \db(S_tA,S_{2t}A)^2
 \leq \frac{3m\delta}{t}
      +\frac{6\delta^2}{t^2}.
 \label{eq:dyadic-increment}
\end{equation}
For $t=2^n$, the distance in \eqref{eq:dyadic-increment} is bounded by a constant times $2^{-n/2}+2^{-n}$. The sum of these bounds is finite, so $(S_{2^n}A)$ is a Cauchy sequence in $(\C,\db)$. By completeness (Lemma~\ref{lem:bures-complete}), it converges to an element $\Phi(A)\in\C$; this shows that $\Phi$ maps $\C$ into $\C$.

For $A,B\in\C$, the approximate isometry property of $S_t$ gives
\[
 \bigl|\db(S_tA,S_tB)-\db(A,B)\bigr|\leq\frac\varepsilon t.
\]
We let $t=2^n\to\infty$ and use continuity of the metric to obtain $\db(\Phi(A),\Phi(B))=\db(A,B)$; in other words, $\Phi$ is an exact isometry (onto its image). Also $S_t(0)=t^{-2}S(0)$, and therefore $\db(S_t(0),0)=\alpha/t\to0$, showing that $\Phi(0)=0$. \end{proof}

Our current goal is to prove Theorem~\ref{thm:bures-main}: given a surjective $\varepsilon$-isometry $S$, we seek a unitary or antiunitary $U$ such that $\db(S(A),UAU^*)\leq2\sqrt2\,\varepsilon$ for every $A\in\C$. The exact isometry $\Phi$ constructed above is our candidate for the conjugation $A\mapsto UAU^*$. We must first identify $\Phi$ as such a conjugation, and then prove the required bound on $\db(S(A),\Phi(A))$.

For the first step, we use Theorem~\ref{thm:MT}, which applies to bijective isometries. Since $\Phi$ preserves distances, it is already injective, so surjectivity is the only remaining thing to prove. This is where we use the assumption that $S$ is onto: we choose a right inverse of $S$ by selecting one preimage of each point, then show that its dyadic rescalings converge to a right inverse of $\Phi$. This will prove that $\Phi$ is onto and allow us to apply Theorem~\ref{thm:MT}.

\begin{lemma}\label{lem:phi-onto} Let $\C:= S_1(H)_+$. If $S : (\C,\db)\to (\C,\db)$ is a surjective $\varepsilon$-isometry, then the map $\Phi : \C\to\C$ in Lemma \ref{lem:blowdown} is bijective. As a result, there is a unitary or antiunitary $U : H\to H$ such that
\begin{equation}
 \Phi(A)=UAU^*.
 \label{eq:phi-wigner}
\end{equation}
\end{lemma}

\begin{proof} It suffices to prove that $\Phi$ is surjective ($\Phi$ is already injective because it preserves distances). Since $S$ is surjective, it already has preimages; the idea is to blow down those preimages at the same time that we blow down $S$.

Since $S$ is surjective, for each $B\in\C$ choose an element $W(B)\in\C$ such that $S(W(B))=B$. This defines a right inverse $W:\C\to\C$ of $S$. Our next goal is to apply Lemma~\ref{lem:blowdown} to $W$, but in order to apply it, we will first need to show that $W$ is an $\varepsilon$-isometry.

Recall also that $S$ is assumed to be an $\varepsilon$-isometry; applying the approximate isometry inequality for $S$ to $W(X)$ and $W(Y)$, and $S(W(X))=X$ and $S(W(Y))=Y$, gives
\[
 \bigl|\db(W(X),W(Y))-\db(X,Y)\bigr|\leq\varepsilon,
 \qquad X,Y\in\C.
\]
Thus $W$ is itself an $\varepsilon$-isometry. For $t>0$, we define its rescaled maps by
\[
 W_t(B) := t^{-2}W(t^2B).
\]
Lemma~\ref{lem:blowdown}, applied to $W$, shows that
\[
 \Psi(B):=\lim_{n\to\infty}W_{2^n}(B)
\]
exists in the Bures metric for every $B\in\C$. Given any $B\in\C$ we will show that $\Psi B$ is a preimage of $B$ under $\Phi$. Since $S\circ W=\mathrm{id}$, the rescaled maps satisfy
\[
 S_t(W_tB)=t^{-2}S(W(t^2B))=B.
\]
We compare this identity with what happens at the fixed input $\Psi B$.

As shown earlier, $S_t$ is an $(\varepsilon/t)$-isometry, and hence
\[
 \db(S_t(\Psi B),B)
 =\db(S_t(\Psi B),S_t(W_tB))
 \leq\db(\Psi B,W_tB)+\frac{\varepsilon}{t}.
\]
Along $t=2^n\to\infty$, we have $W_tB\to\Psi B$, so the right side tends to zero and so $S_t(\Psi B)\to B$. Since $\Psi B$ is fixed as $t$ varies, the definition of $\Phi$ also gives
\[
 S_t(\Psi B)\to\Phi(\Psi B).
\]
Uniqueness of the limit gives $\Phi(\Psi B)=B$.

Since $\Phi$ is a bijective isometry, an immediate application of Molnár--Timmermann's Theorem~\ref{thm:MT} gives that $\Phi(A)$ must equal $UAU^*$ for some $U$ that is either unitary or antiunitary. \end{proof}

\subsection{Comparison with the exact isometry}\label{sec:profiles}

The limit $\Phi$ describes $S$ on large scalar multiples, and recall that when $S$ is surjective, then $\Phi(D)=UDU^*$ is the symmetry that we want to show approximates $S$. To compare the distance between $S(D)$ and $\Phi(D)$ for every fixed $D \in \mathcal C$, we examine what remains of these large distances after subtracting their common growing part---this latter part is the idea behind a Busemann function, which makes sense in any metric space containing a geodesic ray.\footnote{The reader may have encountered Busemann functions in hyperbolic geometry, where the level sets of Busemann functions in the hyperbolic plane are horocycles.} For a unit speed geodesic ray $\gamma$ starting at $o$, a Busemann function is defined by
\[ b_\gamma(x)=\lim_{t\to\infty} \bigl(d(x,\gamma(t))-t\bigr). \]
Here $t=d(o,\gamma(t))$, so the expression inside the limit is equal to $d(x,\gamma(t))-d(o,\gamma(t))$. A quick verification by the reader would show that the triangle inequality makes the expression inside the limit nonincreasing and bounded below, so that the limit must exist. For background on Busemann functions, see \cite[Section~5.3.3]{BBI}.

For comparison, in a real Hilbert space,
\[
 \|tv-x\|-t\|v\|
 \to -\left\langle\frac{v}{\|v\|},x\right\rangle
 \qquad(v\ne0).
\]
Taking the difference of these limits for $x$ and $y$, then the supremum over $v$, recovers $\|x-y\|$. For the Bures metric, distances to $t^2B$ similarly record fidelity with $B$ after the growing term is removed. The following lemma also allows $B$ to vary with $t$, as needed for $S(t^2B)=t^2S_t(B)$. This lemma says that, essentially, fidelity appears as the finite remainder of a distance to a point travelling towards infinity---which is helpful for our proof because we already know how $S$ behaves on large multiples through the limiting isometry $\Phi$.

\begin{lemma}\label{lem:bures-busemann} Let $B_t,B,D\in\C$, suppose $B_t\to B$ in Bures distance, and assume $B\neq0$. Then
\begin{equation}
 \lim_{t\to\infty}
 \left(\db(t^2B_t,D)-t\sqrt{\Tr B_t}\right)
 =-\frac{\Fid(B,D)}{\sqrt{\Tr B}}.
 \label{eq:bures-busemann}
\end{equation}
The same assertion holds along any sequence $t\to\infty$. \end{lemma}

\begin{proof} We put $b_t=\Tr B_t$ and $f_t=\Fid(B_t,D)$. Rationalization gives
\[
 \db(t^2B_t,D)-t\sqrt{b_t}
 =\frac{\Tr D-2tf_t}
 {\db(t^2B_t,D)+t\sqrt{b_t}}.
\]
We first prove the continuity estimate which is needed here. Given $C,C'\in\C$, we choose amplitudes $X,X'$ with $\|X-X'\|_2$ arbitrarily close to $\db(C,C')$. By \eqref{eq:amplitude-fidelity}, $\Fid(C,D)$ and $\Fid(C',D)$ are the suprema, over the same set of amplitudes $Y$ of $D$, of
\[
 \operatorname{Re}\Tr(X^*Y)
 \quad\hbox{and}\quad
 \operatorname{Re}\Tr(X'^*Y).
\]
For real functions $f,g$ on the same set, $|\sup f-\sup g|\leq\sup|f-g|$. Also $\|Y\|_2=\sqrt{\Tr D}$. By the Hilbert-Schmidt Cauchy-Schwarz inequality,
\begin{equation}
 |\Fid(C,D)-\Fid(C',D)|
 \leq\sqrt{\Tr D}\,\db(C,C').
 \label{eq:fidelity-lipschitz}
\end{equation}
Since $B_t\to B$, the reverse triangle inequality gives $\sqrt{b_t}=\db(B_t,0)\to\db(B,0)=\sqrt{\Tr B}$. Equation \eqref{eq:fidelity-lipschitz} also gives $f_t\to\Fid(B,D)$. Finally,
\[
 \left|\db(t^2B_t,D)-t\sqrt{b_t}\right|
 \leq\db(D,0)=\sqrt{\Tr D}.
\]
After the displayed fraction is divided in the numerator and denominator by $t$, its numerator tends to $-2\Fid(B,D)$ and its denominator tends to $2\sqrt{\Tr B}$. \end{proof}

Having established the previous lemma, allowing us to turn distance comparisons (with large multiples) into comparisons of fidelities, our next task is to turn bounds on differences of fidelities into a bound on $\db(A,C)$. For $A,C\in\C$, let
\begin{equation}
 \Delta(A,C):=\sup_{\substack{B\in\C\\ \Tr B=1}}
 |\Fid(B,A)-\Fid(B,C)|.
 \label{eq:profile}
\end{equation}
The authors in \cite{MZC} studied a related quantity using squared fidelity on finite dimensional state spaces.

\begin{lemma}\label{lem:profile} For all $A,C\in\C$,
\begin{equation}
 \frac1{\sqrt2}\db(A,C)\leq\Delta(A,C)\leq\db(A,C).
 \label{eq:profile-inequality}
\end{equation}
If $\dim H\geq2$, the factor $\sqrt2$ in the left inequality is optimal. \end{lemma}

\begin{proof} The upper bound follows from symmetry of fidelity and \eqref{eq:fidelity-lipschitz}: for every state $B$,
\[
 |\Fid(B,A)-\Fid(B,C)|
 =|\Fid(A,B)-\Fid(C,B)|\leq\db(A,C).
\]

The lower bound is immediate if $A=C$, so assume $A\ne C$. First suppose that $H$ is finite dimensional and $A$ is invertible. We seek amplitudes $X=A^{1/2}$ and $Y=TA^{1/2}$ with $T\geq0$. Then $X^*Y=A^{1/2}TA^{1/2}\geq0$, so these amplitudes will attain the Bures distance provided that $YY^*=C$. This last condition is $TAT=C$, or equivalently
\[
 (A^{1/2}TA^{1/2})^2=A^{1/2}CA^{1/2}.
\]
Taking the positive square root gives
\begin{equation}
 T=A^{-1/2}(A^{1/2}CA^{1/2})^{1/2}A^{-1/2}.
 \label{eq:transport-map}
\end{equation}
This is the usual transport operator for the Bures--Wasserstein metric; see \cite[Section~3]{BJL}. For this choice, $YY^*=TAT=C$, and the positivity of $X^*Y$ gives
\begin{equation}
 \db(A,C)^2=\|X-Y\|_2^2=\Tr A(I-T)^2.
 \label{eq:transport-distance}
\end{equation}
An amplitude $Z$ with $\|Z\|_2=1$ defines a state $ZZ^*$. When $A\ne C$, Cauchy--Schwarz suggests choosing $Z$ proportional to $X-Y$ to make $\operatorname{Re}\Tr Z^*(X-Y)$ large. However, this choice need not make both $Z^*X$ and $Z^*Y$ positive, so their traces need not equal the corresponding fidelities. This motivates us to treat the positive and negative parts of $I-T$ separately. For a self-adjoint operator $R$, we write $R_+=(|R|+R)/2$ for its positive part. We let
\[
 R_+=(I-T)_+,\quad p=\Tr AR_+^2,\qquad
 R_-=(T-I)_+,\quad q=\Tr AR_-^2.
\]
The operators $R_+$ and $R_-$ are functions of $T$ and so they commute with $T$; their product is zero, and $I-T=R_+-R_-$. It follows from \eqref{eq:transport-distance} that
\begin{equation}
 p+q=\Tr A(I-T)^2=\db(A,C)^2.
 \label{eq:pq-bures}
\end{equation}
If $p>0$, then $\|R_+A^{1/2}\|_2^2=p$. Normalizing this operator gives a unit amplitude and hence a state:
\[
 Z=\frac{R_+A^{1/2}}{\sqrt p},\qquad
 B_+=ZZ^*=\frac{R_+AR_+}{p}.
\]
We have
\[
 Z^*X=\frac{A^{1/2}R_+A^{1/2}}{\sqrt p}\geq0.
\]
Moreover $R_+T\geq0$, because $R_+$ commutes with the positive operator $T$, and hence
\[
 Z^*Y=\frac{A^{1/2}R_+TA^{1/2}}{\sqrt p}\geq0.
\]
The equality case following \eqref{eq:amplitude-fidelity} now gives
\[
 \Fid(B_+,A)=\Tr(Z^*X),\qquad
 \Fid(B_+,C)=\Tr(Z^*Y).
\]
We have $R_+(I-T)=R_+^2$, so as a result,
\begin{equation}
 \Fid(B_+,A)-\Fid(B_+,C)
 =\Tr Z^*(X-Y)
 =\frac{\Tr(AR_+^2)}{\sqrt p}=\sqrt p.
 \label{eq:positive-profile}
\end{equation}
The same argument, on the range where $T-I$ is positive, shows that if $q>0$, the state $B_-=R_-AR_-/q$ satisfies
\begin{equation}
 \Fid(B_-,C)-\Fid(B_-,A)=\sqrt q.
 \label{eq:negative-profile}
\end{equation}
By \eqref{eq:pq-bures}, at least one of $p,q$ is at least $\db(A,C)^2/2$. This proves the lower bound when $A>0$ and $H$ is finite dimensional.

Still in finite dimension, the case where $A$ is singular follows by replacing it with $A_\eta=A+\eta I$ and letting $\eta\to0^+$. Indeed, $A_\eta\to A$ in trace norm and hence in Bures distance by Lemma \ref{lem:bures-complete}. Equation \eqref{eq:fidelity-lipschitz}, together with symmetry of fidelity, shows that
\[
 |\Delta(A_\eta,C)-\Delta(A,C)|\leq\db(A_\eta,A).
\]
We may therefore pass to the limit in the inequality already proved for $A_\eta$.

We now suppose that $H$ is infinite dimensional and work on the closed span $M$ of the ranges of $A$ and $C$. The space $M$ is separable because the range of every trace-class operator is separable. We choose finite rank projections $P_n$ on $M$ which increase strongly to $I_M$. Finite rank operators are dense in the trace-class, so
\[
 A_n=P_nAP_n\to A,\qquad
 C_n=P_nCP_n\to C
\]
in trace norm and hence in Bures distance. If a state $B$ is supported on $P_nH$, then $B^{1/2}=P_nB^{1/2} = B^{1/2} P_n$ and
\[
 B^{1/2}AB^{1/2}=B^{1/2}A_nB^{1/2}.
\]
Thus $\Fid(B,A)=\Fid(B,A_n)$, and the same is true for $C$. The finite dimensional result on $P_nH$ then gives
\[
 \Delta(A,C)\geq\frac1{\sqrt2}\db(A_n,C_n).
\]
We let $n\to\infty$ to obtain the lower bound.

We see that the constant cannot be improved by taking $A=aP$ and $C=aQ$ with $a>0$ and orthogonal rank one projections $P,Q$. The orthogonality of $P$ and $Q$ gives $\Fid(A,C)=0$, so $\db(A,C)=\sqrt{2a}$. For any state $B$,
\[
 \Fid(B,aP)=\sqrt{a\Tr(BP)},\qquad
 \Fid(B,aQ)=\sqrt{a\Tr(BQ)}.
\]
Since $0\leq\Tr(BP),\Tr(BQ)\leq1$, both fidelities lie in $[0,\sqrt a]$, and their absolute difference is at most $\sqrt a$. Taking $B=P$ gives fidelities $\sqrt a$ and $0$, respectively, so this bound is attained, letting one conclude
\[
 \Delta(A,C)=\sqrt a=\frac1{\sqrt2}\db(A,C),
\]
which proves that the constant is optimal. \end{proof}

\begin{proof}[Proof of Theorem \ref{thm:bures-main}] We let $\Phi(A)=UAU^*$ be the exact isometry from Lemma \ref{lem:phi-onto}. By surjectivity of $S$, we may choose $X_0\in\C$ such that $S(X_0)=0$.

We fix $B\neq0$ and $D,E\in\C$. We apply the approximate isometry inequality to the pairs $(t^2B,D)$ and $(t^2B,E)$:
\begin{align}
 \Big|&\big[\db(S(t^2B),S(D))-\db(S(t^2B),S(E))\big]
 \notag\\
 &-\big[\db(t^2B,D)-\db(t^2B,E)\big]\Big|
 \leq2\varepsilon.
 \label{eq:two-distance-comparison}
\end{align}
Subtracting the two distances in each bracket cancels their common growing term. Each square bracket can be written as the difference of two expressions of the form appearing in Lemma \ref{lem:bures-busemann}. For the first bracket we use $S(t^2B)=t^2S_t(B)$ and $S_t(B)\to\Phi(B)$ along $t=2^n\to\infty$. For the second we use the constant family $B_t=B$. Since $\Phi$ fixes zero and is an isometry,
\[
 \sqrt{\Tr\Phi(B)}=\db(\Phi(B),0)=\db(B,0)=\sqrt{\Tr B}.
\]
We pass to the limit along $t=2^n\to\infty$ in \eqref{eq:two-distance-comparison}. If
\begin{equation}
 e_D(B)=\Fid(\Phi(B),S(D))-\Fid(B,D),
 \label{eq:fidelity-error}
\end{equation}
then the result is
\begin{equation}
 |e_D(B)-e_E(B)|\leq2\varepsilon\sqrt{\Tr B}.
 \label{eq:fidelity-error-difference}
\end{equation}

Estimate \eqref{eq:fidelity-error-difference} bounds differences of errors. To bound an individual error, we compare it with a nonnegative error and a nonpositive error, supplied by 0 and $X_0$ respectively. Since $\Fid(B,0)=0$,
\[
 e_0(B)=\Fid(\Phi(B),S(0))\geq0.
\]
On the other hand, $S(X_0)=0$, so
\[
 e_{X_0}(B)=-\Fid(B,X_0)\leq0.
\]
Using $E=0$ in \eqref{eq:fidelity-error-difference}, we obtain
\[
 e_D(B)\geq e_0(B)-2\varepsilon\sqrt{\Tr B}
 \geq-2\varepsilon\sqrt{\Tr B}.
\]
Using $E=X_0$ gives
\[
 e_D(B)\leq e_{X_0}(B)+2\varepsilon\sqrt{\Tr B}
 \leq2\varepsilon\sqrt{\Tr B}.
\]
Thus
\begin{equation}
 |\Fid(\Phi(B),S(D))-\Fid(B,D)|
 \leq2\varepsilon\sqrt{\Tr B}.
 \label{eq:profile-control}
\end{equation}
Since $\Phi$ preserves fidelity, $\Fid(B,D)=\Fid(\Phi(B),\Phi(D))$. Moreover, $\Phi$ maps the state space onto itself. Taking the supremum in \eqref{eq:profile-control} over states $B$ gives
\[
 \Delta(S(D),\Phi(D))\leq2\varepsilon.
\]
Lemma~\ref{lem:profile} yields
\[
 \db(S(D),UDU^*)=\db(S(D),\Phi(D))
 \leq2\sqrt2 \varepsilon.
\]

It remains to prove the assertions concerning zero. Let us use the notation
\[
 P=S(0),\qquad Q=\Phi(X_0),\qquad
 a=\sqrt{\Tr P},\qquad b=\sqrt{\Tr Q}.
\]
Since $\Phi$ preserves trace, $b=\db(X_0,0)$. By equation~\eqref{eq:fidelity-error-difference}, with $D=0$ and $E=X_0$,
\[
 \Fid(C,P)+\Fid(C,Q)\leq2\varepsilon,
 \qquad C\in\St,
\]
because $\Phi$ preserves fidelity and maps $\St$ onto itself.

If $a^2+b^2=0$, both assertions are immediate. Assume now that $a^2+b^2\ne0$; since $P^{1/2}(P+Q)P^{1/2}\geq P^2$, monotonicity of the positive square root and symmetry of fidelity give $\Fid(P+Q,P)\geq\Tr P$ and similarly, $\Fid(P+Q,Q)\geq\Tr Q$, so that
\[
 2\varepsilon
 \geq\Fid(C,P)+\Fid(C,Q)
 \geq\frac{\Tr P+\Tr Q}{\sqrt{a^2+b^2}}
 =\sqrt{a^2+b^2}, \qquad  C=\frac{P+Q}{a^2+b^2}.
\]
The approximate isometry inequality for $(0,X_0)$ also gives $|a-b|\leq\varepsilon$.

If $M:=\max\{a,b\}>\varepsilon$, then these inequalities imply
\[
 M^2+(M-\varepsilon)^2\leq4\varepsilon^2,
\]
and so
\[
 M\leq\frac{1+\sqrt7}{2} \varepsilon.
\]
The same bound is immediate when $M\leq\varepsilon$. The argument applies to every preimage $X_0$ of zero, which proves both assertions. \end{proof}

\begin{remark} The factor $2$ comes from comparing two distances, each with error at most $\varepsilon$. The factor $\sqrt2$ comes from Lemma~\ref{lem:profile}; its orthogonal projection example shows that this factor is necessary. (These facts are about the profile, and do not determine the best constant in Theorem~\ref{thm:bures-main}.) \end{remark}

\begin{remark}\label{rem:bures-trace-preserving} Under the assumptions of Theorem~\ref{thm:bures-main}, if $S$ also preserves trace, then
\[
 \db(S(D),UDU^*)\leq\sqrt2 \varepsilon,
 \qquad D\in\C.
\]
Indeed, since $S$ preserves trace, $\Tr S(t^2B) = \Tr (t^2B) = t^2$ for $B\in\St$. Subtracting $t$ from both distances in the $\varepsilon$-isometry inequality for $t^2B,D$, and applying Lemma~\ref{lem:bures-busemann} along $t=2^n\to\infty$, gives
\[
 |\Fid(\Phi(B),S(D))-\Fid(B,D)|\leq\varepsilon.
\]
Since $\Phi$ preserves fidelity and maps $\St$ onto itself, taking the supremum over $B\in\St$ yields $\Delta(S(D),\Phi(D))\leq\varepsilon$. Lemma~\ref{lem:profile} gives the claimed bound. \end{remark}

\begin{corollary}\label{cor:bures-delta-surjective}
Let $\varepsilon,\delta\geq0$, and let $S:(\C,\db)\to(\C,\db)$ be a $\delta$-surjective $\varepsilon$-isometry with $S(0)=0$. There is a unitary or antiunitary $U$ on $H$ such that
\[
 \db(S(A),UAU^*)\leq2\sqrt2\,\varepsilon,
 \qquad A\in\C.
\]
In particular, the bound does not depend on $\delta$.
\end{corollary}

\begin{proof} We choose $W:\C\to\C$ with $\db(S(W(B)),B)\leq\delta$. The triangle inequality shows that $W$ is an $(\varepsilon+2\delta)$-isometry. Applying Lemma~\ref{lem:blowdown} to $S$ and $W$, we write $\Phi$ and $\Psi$ for their respective dyadic limits. With the notation of Lemma~\ref{lem:phi-onto}, we have
\[
 \db(S_t(\Psi B),B)
 \leq\db(\Psi B,W_tB)+\frac{\varepsilon+\delta}{t}.
\]
We let $t=2^n\to\infty$ to obtain $\Phi(\Psi B)=B$. Thus $\Phi$ is surjective, and Theorem~\ref{thm:MT} gives $\Phi(A)=UAU^*$. The argument giving~\eqref{eq:fidelity-error-difference} still applies. Since $S(0)=0$, we have $e_0(B)=0$, so taking $E=0$ gives
\[
 |\Fid(\Phi(B),S(D))-\Fid(B,D)|
 \leq2\varepsilon\sqrt{\Tr B}.
\]
Since $\Phi$ preserves fidelity and maps the state space onto itself, taking the supremum over states $B$ and applying Lemma~\ref{lem:profile} proves the bound.
\end{proof}

\section{Trace norm stability on the cone}

For the trace metric, we combine the operator norm estimate of \cite[Theorem~3.15]{DLLZ} with the following theorem. Only the finiteness of that estimate is needed; its numerical constant does not enter the trace norm bound.

\begin{theorem}\label{thm:trace-rigidity} Let $S:(\C,\done)\to(\C,\done)$ be an $\varepsilon$-isometry. If
\begin{equation}
 \sup_{A\in\C}\|S(A)-A\|_\infty<\infty,
 \label{eq:operator-bound}
\end{equation}
then
\begin{equation}
 \bigl\|(S(D)-S(E))-(D-E)\bigr\|_1\leq3\varepsilon,
 \qquad D,E\in\C.
 \label{eq:trace-difference}
\end{equation}
\end{theorem}

Neither surjectivity nor preservation of zero is assumed here. If $S(0)=0$, taking $E=0$ gives the bound $3\varepsilon$ in the standard case. The $\delta$-surjective corollary also uses the passage from $\delta$-surjectivity to surjectivity in \cite[Lemma~5.1]{DLL}, recalled as \cite[Lemma~2.7]{DLLZ}.

\subsection{Two trace norm limits}\label{sec:trace-lemmas}

We first show that $S(nB)/n$ converges to $B$ in trace norm. We then compute the finite term left in $\|tB-D\|_1$ after subtracting $t\Tr B$; this term will recover the sign pairing used in the next section.

\begin{lemma}\label{lem:trace-scaling} Under the assumptions of Theorem \ref{thm:trace-rigidity},
\begin{equation}
 \frac{S(nB)}n\to B
 \quad\hbox{in }S_1(H)
 \label{eq:trace-scaling}
\end{equation}
for every $B\in\C$. \end{lemma}

\begin{proof} We let $M_0$ denote the supremum in \eqref{eq:operator-bound}, and put $B_n=S(nB)/n$ and $P_0=S(0)$. Then
\begin{equation}
 \|B_n-B\|_\infty\leq\frac{M_0}{n},
 \qquad
 |\Tr B_n-\Tr B|\leq\frac{\varepsilon+\Tr P_0}{n}.
 \label{eq:operator-trace-convergence}
\end{equation}
The first estimate follows directly from \eqref{eq:operator-bound}. For the second, the approximate isometry inequality gives
\[
 \bigl|\|S(nB)-P_0\|_1-n\Tr B\bigr|\leq\varepsilon.
\]
Since $S(nB)\geq0$, the reverse triangle inequality gives
\[
 \bigl|\Tr S(nB)-\|S(nB)-P_0\|_1\bigr|\leq\Tr P_0.
\]
Combining this with the preceding estimate and dividing by $n$ proves the second inequality in \eqref{eq:operator-trace-convergence}.

In infinite dimension, operator norm convergence alone is insufficient: a rank $n$ projection divided by $n$ has operator norm $1/n$ but trace norm $1$. The convergence of traces supplies the additional control. We first work on a finite dimensional subspace containing almost all of the limit's trace. Operator norm convergence controls the compression to this subspace, while convergence of traces controls the trace outside it. Suppose that $T_j,T \in \C$ (in particular, both are positive), that $T_j\to T$ in operator norm, and that $\Tr T_j\to\Tr T$. We claim that $T_j\to T$ in trace norm.

We let $F$ be a finite rank spectral projection of $T$. Since $F$ has finite rank,
\[
 \|F(T_j-T)F\|_1
 \leq\operatorname{rank}(F)\|T_j-T\|_\infty\to0.
\]
It follows that $\Tr(FT_j)\to\Tr(FT)$. Together with convergence of the full traces, this gives
\[
 \Tr(F^\perp T_j)=\Tr T_j-\Tr(FT_j)\to
 \Tr T-\Tr(FT)=\Tr(F^\perp T).
\]

We need one estimate for the part outside $F$. If $X\geq0$ is trace-class, then
\[
 X-FXF=FXF^\perp+F^\perp XF+F^\perp XF^\perp.
\]
By \eqref{eq:schatten-holder},
\[
 \|FXF^\perp\|_1
 =\|FX^{1/2}X^{1/2}F^\perp\|_1
 \leq\sqrt{\Tr(FX)\Tr(F^\perp X)}.
\]
The other off diagonal term has the same norm, while $\|F^\perp XF^\perp\|_1=\Tr(F^\perp X)$. Therefore
\begin{equation}
 \|X-FXF\|_1
 \leq2\sqrt{\Tr X\,\Tr(F^\perp X)}+\Tr(F^\perp X).
 \label{eq:tail-estimate}
\end{equation}
Because $T$ is positive and trace-class, we can choose $F$ so that $\Tr(F^\perp T)$ is as small as desired. The convergence of the preceding tail traces makes $\Tr(F^\perp T_j)$ small for all sufficiently large $j$.

We now use
\[
 \|T_j-T\|_1
 \leq\|T_j-FT_jF\|_1+\|F(T_j-T)F\|_1+\|T-FTF\|_1.
\]
For fixed $F$, the middle term tends to zero because $F$ has finite rank and $T_j\to T$ in operator norm. Applying \eqref{eq:tail-estimate} to the other two terms, and using $\Tr T_j\to\Tr T$ and $\Tr(F^\perp T_j)\to\Tr(F^\perp T)$, gives
\[
 \limsup_{j\to\infty}\|T_j-T\|_1
 \leq4\sqrt{\Tr T\,\Tr(F^\perp T)}
      +2\Tr(F^\perp T).
\]
We can make the right side arbitrarily small by choosing $F$ so that $\Tr(F^\perp T)$ is sufficiently small. Thus $T_j\to T$ in trace norm.

We now apply this result with $T_j=B_j=S(jB)/j$ and $T=B$. These operators are positive and trace-class, and \eqref{eq:operator-trace-convergence} gives both $\|B_j-B\|_\infty\to0$ and $\Tr B_j\to\Tr B$. Consequently, $\|B_j-B\|_1\to0$, which is precisely \eqref{eq:trace-scaling}. \end{proof}

\begin{lemma}\label{lem:trace-busemann} If $B,D\in\C$, then
\begin{equation}
 \lim_{t\to\infty}\bigl(\|tB-D\|_1-t\Tr B\bigr)
 =\Tr D-2\Tr(s(B)D).
 \label{eq:trace-busemann}
\end{equation}
\end{lemma}

\begin{proof} For scalars $b,d\geq0$,
\[
 \lim_{t\to\infty}(|tb-d|-tb)
 =
 \begin{cases}
  -d,&b>0,\\
  d,&b=0.
 \end{cases}
\]
For diagonal operators, this subtracts the entries of $D$ where $B$ is nonzero and adds them where $B$ is zero, giving $\Tr D-2\Tr(s(B)D)$. We now prove that the same formula holds without a commutativity assumption.

We regard $S_1(H)_{\mathrm{sa}}$ as a real Banach space. With $h=1/t$, the left side is the one sided directional derivative
\[
 \lim_{h\to0^+}\frac{\|B-hD\|_1-\|B\|_1}{h}.
\]
We explain the norm derivative which is used here. Trace norm duality \eqref{eq:trace-duality} shows that the derivative is
\begin{equation}
 \max\{-\Tr(KD):K=K^*,\ \|K\|_\infty\leq1,\
                         \Tr(KB)=\|B\|_1\}.
 \label{eq:subdifferential}
\end{equation}
Indeed, we fix any $K$ in the displayed set, then
\[
 \|B-hD\|_1\geq\Tr(K(B-hD))
 =\|B\|_1-h\Tr(KD),
\]
which gives the lower bound for the derivative. For the reverse bound, for each $h>0$ choose a self-adjoint contraction $K_h$ which norms the self-adjoint operator $B-hD$:
\[
 \Tr(K_h(B-hD))=\|B-hD\|_1.
\]
For example, we may take $K_h$ to be the sign of $B-hD$, with either sign on its kernel. On the closed span of the ranges of $B$ and $D$, a sequence $K_{h_j}$ with $h_j\to0^+$ has a weak star convergent subsequence by applying Banach-Alaoglu together with separability of that closed span of the ranges of $B$ and $D$. If its limit is $K$, then
\[
 \Tr(KB)=\lim_j\Tr(K_{h_j}B)=\Tr B,
\]
because $\|B-h_jD\|_1\to\|B\|_1$ and $h_j\Tr(K_{h_j}D)\to0$. Also $\Tr(K_{h_j}B)\leq\Tr B$, so
\[
 \frac{\|B-h_jD\|_1-\|B\|_1}{h_j}
 \leq-\Tr(K_{h_j}D)\to-\Tr(KD).
\]
We take a sequence which realizes the upper limit and obtain \eqref{eq:subdifferential}. In particular, the one sided derivative exists.

If $K$ occurs in \eqref{eq:subdifferential}, then $I-K\geq0$ and
\[
 0=\Tr B-\Tr(KB)
 =\Tr\bigl(B^{1/2}(I-K)B^{1/2}\bigr).
\]
The operator inside the trace is positive. A positive trace-class operator with trace zero is zero, so $(I-K)^{1/2}B^{1/2}=0$. The range of $B^{1/2}$ is dense in $s(B)H$, so $K=I$ on $s(B)H$. Since $K$ is self-adjoint, it has no off diagonal part between $s(B)H$ and $s(B)^\perp H$, and hence
\[
 K=I_{s(B)H}\oplus K_0,
 \qquad K_0=K_0^*,\quad \|K_0\|_\infty\leq1,
\]
on $s(B)H\oplus s(B)^\perp H$. The maximum in \eqref{eq:subdifferential} is attained at $K_0=-I$ and its value is
\[
 -\Tr(s(B)D)+\Tr(s(B)^\perp D)=\Tr D-2\Tr(s(B)D),
\]
which proves the lemma. \end{proof}

\subsection{From operator norm to trace norm}\label{sec:trace-rigidity}

We now prove Theorem~\ref{thm:trace-rigidity}. The defining inequality in the definition of an $\varepsilon$-isometry is a bound on distance comparisons, so we need to express the trace norm error through distance comparisons.

Lemma~\ref{lem:trace-busemann} gives, for fixed $B,D\in \C$,
\[
 \lim_{t\to\infty}\bigl(\|tB-D\|_1-t\Tr B\bigr)
 =\Tr D-2\Tr(RD)
 =\Tr((I-R)D(I-R))-\Tr(RDR), \quad R:=s(B).
\]
The last equality decomposes the trace over $H=RH\oplus(I-R)H$. Thus, after subtracting the growing term $t\Tr B$, the distance converges to a fixed number: the trace of $D$ outside the support of $B$, minus its trace inside that support.

This is useful because the trace norm is obtained by assigning a minus sign to negative eigenvalues. More precisely, if $X$ is self-adjoint and trace-class, and $R$ projects onto the closed span of its eigenvectors with negative eigenvalues, then
\[
 \Tr((I-R)X(I-R))-\Tr(RXR)=\|X\|_1.
\]
This tells us how to choose the supports of the distant operators: their signs should match those of the error we want to estimate.

The lemma applies to positive operators, but its limiting expression is linear in $D$. The proof below combines distance comparisons to bound this expression on the self-adjoint error. It uses two positive operators with complementary supports on the subspace under consideration, together with Lemma~\ref{lem:trace-scaling} to control their images under $S$.

\begin{proof}[Proof of Theorem~\ref{thm:trace-rigidity}] For $D,E\in\C$, let
\[
G:=(S(D)-S(E))-(D-E).
\]
We must prove that $\|G\|_1\leq3\varepsilon$. To do this, we will express $\|G\|_1$ as a trace $\Tr(JG)$ and estimate that trace using distances, to which the approximate isometry inequality applies.

Let $M$ be the closed linear span of the supports of $D,E,S(D)$, and $S(E)$. This subspace is separable, even when $H$ is not, because trace-class operators have separable ranges. All four operators vanish on $M^\perp$. If $H$ is separable, we may take $M=H$.

We define $J$ on $M$ to act as $+1$ on the positive spectral subspace of $G$ and as $-1$ on its negative spectral subspace. On $\ker G\cap M$, we let $J$ act as $+1$. This last choice makes $J$ a self-adjoint unitary without affecting the product $JG$. Multiplication by $J$ changes each negative eigenvalue of $G$ to its absolute value, so $JG=|G|$. Therefore
\begin{equation}
 \Tr(JG)=\|G\|_1.
 \label{eq:norming-symmetry}
\end{equation}

Let
\[
 P:=\frac{I_M+J}{2},\qquad Q:=\frac{I_M-J}{2}.
\]
These are the orthogonal projections onto the subspaces where $J$ acts as $+1$ and $-1$, respectively. In particular,
\[
 P+Q=I_M,\qquad J=P-Q.
\]

We now choose the directions of our two rays. In Lemma~\ref{lem:trace-busemann}, a direction with support $R$ gives the limiting expression $\Tr X-2\Tr(RX)$. For $X$ supported on $M$, the identities
\[
 I_M-2P=-J,\qquad I_M-2Q=J
\]
show that supports $P$ and $Q$ give $-\Tr(JX)$ and $\Tr(JX)$, respectively. Hence we will choose positive trace-class operators $A,C$, supported on $M$, such that
\begin{equation}
 s(A)=P,\qquad s(C)=Q.
 \label{eq:faithful-probes}
\end{equation}
If $P=0$, we put $A=0$, and if $Q=0$, we put $C=0$. To construct operators $A$ and $C$, recall that $M$ is separable, and so choose an orthonormal basis $(u_j)$ of $PM$, with
\[ A:=\sum_j2^{-j}u_j\otimes u_j. \]
The summability of the coefficients makes $A$ trace-class, and their positivity gives $s(A)=P$. We construct $C$ in the same way on $QM$.

For $n\geq1$, let
\[
X_n=S(nA),\qquad Y_n=S(nC),\qquad Z_n=X_n-Y_n.
\]
We choose a self-adjoint unitary $J_n$ on $H$ extending $\operatorname{sgn}(Z_n)$. Thus $J_nZ_n=|Z_n|$, and
\begin{equation}
\Tr(J_nZ_n)=\|Z_n\|_1.
 \label{eq:norming-Jn}
\end{equation}
Lemma \ref{lem:trace-scaling} gives
\begin{equation}
\frac{Z_n}{n}\to A-C\quad\hbox{in }S_1(H), \qquad \frac{\|Z_n\|_1}{n}\to a+c, \qquad a := \Tr A, \ c := \Tr C.
 \label{eq:Zn-limit}
\end{equation}
For the second limit, we used $\|A-C\|_1=\Tr A+\Tr C=a+c$, since $A$ and $C$ have orthogonal supports.

For a self-adjoint contraction $K$, the largest possible value of $\Tr(K(A-C))$ is $a+c$. To attain $a+c$, we need both $\Tr(KA)=a$ and $\Tr(KC)=-c$; equivalently,
\[ \Tr((I-K)A)=0,\qquad \Tr((I+K)C)=0. \]
Since $I-K$ and $I+K$ are positive, these equalities imply that $K$ acts as $+1$ (i.e. is the identity) on the support of $A$ and as $-1$ on the support of $C$. Thus $K$ agrees with $J$ on $M$.

We will show that $J_n$ approaches this largest value, then use the two nonnegative deficits to prove convergence of its traces against fixed trace-class operators.

Trace norm duality and \eqref{eq:norming-Jn} give
\begin{equation}
\left|\Tr(J_n(A-C))-\frac{\|Z_n\|_1}{n}\right| \leq\left\|A-C-\frac{Z_n}{n}\right\|_1\to0,
 \label{eq:weak-star-identification}
\end{equation}
so $\Tr(J_n(A-C))\to a+c$ by \eqref{eq:Zn-limit}.

We put $P_n=(I+J_n)/2$ and $Q_n=(I-J_n)/2$; these are orthogonal projections on $H$, and
\[
\begin{aligned}
a-\Tr(J_nA)
&=2\Tr\bigl(A^{1/2}Q_nA^{1/2}\bigr)
=2\|Q_nA^{1/2}\|_2^2,\\
c+\Tr(J_nC)
&=2\Tr\bigl(C^{1/2}P_nC^{1/2}\bigr)
=2\|P_nC^{1/2}\|_2^2.
\end{aligned}
\]
Both quantities are nonnegative, and their sum is $a+c-\Tr(J_n(A-C))\to0$. Hence
\begin{equation}
\|Q_nA^{1/2}\|_2\to0,\qquad\|P_nC^{1/2}\|_2\to0.
 \label{eq:Qn-to-zero}
\end{equation}
Equation~\eqref{eq:Qn-to-zero} implies $Q_nA^{1/2}x\to0$ for every $x\in H$. The range of $A^{1/2}$ is dense in $PM$, and $\|Q_n\|_\infty\leq1$. Hence approximating a vector of $PM$ by vectors in this range gives $Q_nx\to0$ for every $x\in PM$. Similarly, $P_nx\to0$ for every $x\in QM$. Thus, for $x\in M$,
\[
(J_n-J)x=-2Q_nPx+2P_nQx\to0.
\]
It follows that
\begin{equation}
\Tr(J_nX)\to\Tr(JX)
 \label{eq:test-SD}
\end{equation}
for every trace-class $X$ supported on $M$: first check this for rank one operators, then use finite rank approximation and $\|J_n\|_\infty=1$. In particular, this applies to $X=S(D)-S(E)$.

Since $\Tr(JG)=\|G\|_1$, it suffices to bound this pairing from above. We compare $S(D)$ and $S(E)$ through $Y_n$ and $X_n$:
\begin{align*}
r_n&:=\Tr(J_n(S(D)-S(E)))\\
 &=\Tr J_n\bigl((S(D)-Y_n)+(Y_n-X_n)+(X_n-S(E))\bigr)\\
 &\leq\|S(D)-Y_n\|_1-\|X_n-Y_n\|_1
       +\|X_n-S(E)\|_1\\
 &\leq\bigl(\|nC-D\|_1-nc\bigr)
       +\bigl(\|nA-E\|_1-na\bigr)+3\varepsilon.
\end{align*}
The middle term is exact because $J_n$ norms $X_n-Y_n$; the last inequality uses the approximate isometry property three times. Since $s(C)=Q$, $s(A)=P$, and $J=P-Q$ on $M$, Lemma~\ref{lem:trace-busemann} gives
\[
\|nC-D\|_1-nc\to\Tr(JD),\qquad
 \|nA-E\|_1-na\to-\Tr(JE).
\]
Using \eqref{eq:test-SD} and passing to the limit yields
\[
\|G\|_1
 =\Tr J\bigl((S(D)-S(E))-(D-E)\bigr)
 \leq3\varepsilon.
\]
This proves \eqref{eq:trace-difference}. \end{proof}

\subsection{Proof of the trace norm stability theorem}\label{sec:trace-main}

\begin{proof}[Proof of Theorem \ref{thm:trace-global}] The operator norm estimate of \cite[Theorem 3.15]{DLLZ} requires the surjective $\eta$-isometry $R$ to obey $R(0)=0$; it gives a unitary or antiunitary $U$ such that $\sup_{A\in \mathcal{C}}\|R(A)-UAU^*\|_\infty \le 11554\eta$.

We temporarily arrange this by swapping the values of $T$ at zero and at a preimage of zero. The modified map is standard and surjective, so \cite[Theorem~3.15]{DLLZ} provides a unitary or antiunitary $U$. Since the swap changes only two values, $U^*T(\cdot)U$ also has bounded operator norm error. We then apply the trace norm rigidity theorem to this conjugate of the original $\varepsilon$-isometry.

Let $P_0:=T(0), b:=\Tr P_0$. By surjectivity, we choose $X_0\in\C$ such that $T(X_0)=0$. If $P_0=0$, let $R=T$. If $P_0\neq0$, then $X_0\neq0$, and we define the map $R$ as follows,
\[
 R(0)=0,\qquad R(X_0)=P_0,\qquad
 R(A)=T(A)\quad(A\neq0,X_0).
\]
This map $R$ is obtained by interchanging the inputs $0$ and $X_0$ of $T$ (and equals $T$ when $P_0=0$); hence it has the same range as $T$, and $R(0)=0$. Moreover, $\|R(A)-T(A)\|_1\leq b$ for every $A$. The triangle inequality shows that $R$ is an $\eta$-isometry, where $ \eta=\varepsilon+2b.$ This size of $\eta$ will not enter the final estimate and is instead used only to obtain a finite operator norm bound.

By \cite[Theorem~3.15]{DLLZ}, there is a unitary or antiunitary $U$ such that
\begin{equation}
 M:=\sup_{A\in\C}\|U^*R(A)U-A\|_\infty<\infty.
 \label{eq:uniform-operator}
\end{equation}
We return to the original map and put
\[
 S(A)=U^*T(A)U.
\]
Since $\|T(A)-R(A)\|_1\leq b$, we have
\[
 \sup_{A\in\C}\|S(A)-A\|_\infty\leq b+M<\infty.
\]
Thus $S$ satisfies the hypothesis of Theorem \ref{thm:trace-rigidity}. It remains an $\varepsilon$-isometry, since conjugation by $U$ preserves trace distance.

We put $Q_0=S(0)=U^*P_0U$ so $\Tr Q_0=b$, and $S(X_0)=0$. Theorem \ref{thm:trace-rigidity}, first with $D=X_0$ and $E=0$, gives
\begin{equation}
 \Tr X_0+\Tr Q_0=\|-X_0-Q_0\|_1\leq3\varepsilon.
 \label{eq:two-basepoint-traces}
\end{equation}
For an arbitrary $A\in\C$, we may use either $E=0$ or $E=X_0$ in \eqref{eq:trace-difference}. These two choices give
\begin{align*}
 \|S(A)-A\|_1
 &\leq3\varepsilon+\Tr Q_0,\\
 \|S(A)-A\|_1
 &\leq3\varepsilon+\Tr X_0.
\end{align*}
As a result,
\[
 \|S(A)-A\|_1
 \leq3\varepsilon+\min\{\Tr Q_0,\Tr X_0\}
 \leq\frac92\varepsilon.
\]
We conjugate back and obtain \eqref{eq:trace-global}. If $T(0)=0$, then $Q_0=0$, and the first of the preceding two estimates gives the sharper bound $3\varepsilon$. Finally, for any $X\in\C$ with $T(X)=0$, Theorem \ref{thm:trace-rigidity}, applied to $D=X$ and $E=0$, finishes the proof of \eqref{eq}:
\[
 \Tr X+\Tr T(0)=\|-X-Q_0\|_1\leq3\varepsilon.
\]
\end{proof}

\begin{remark} Surjectivity is used above to choose a preimage $X_0$ of zero and to apply \cite[Theorem~3.15]{DLLZ} to $R$. The finite operator norm bound then allows us to apply Theorem \ref{thm:trace-rigidity} to the original $\varepsilon$-isometry after conjugation. \end{remark}

\begin{proof}[Proof of Corollary \ref{cor:trace-delta-surjective}] We use a surjective replacement only to identify the conjugation; the final error bound will come from applying Theorem~\ref{thm:trace-rigidity} to the original map. By \cite[Lemma~5.1]{DLL} (see also \cite[Lemma~2.7]{DLLZ}), there are finite constants $\alpha,\beta>0$ and a standard surjective $\alpha$-isometry $R:(\C,\done)\to(\C,\done)$ such that
\[
 \sup_{A\in\C}\|T(A)-R(A)\|_1\leq\beta.
\]
We apply \cite[Theorem~3.15]{DLLZ} to $R$ and obtain a unitary or antiunitary $U$ on $H$ such that
\[
 M:=\sup_{A\in\C}\|U^*R(A)U-A\|_\infty<\infty.
\]
Since the operator norm is bounded by the trace norm, the preceding estimates give
\[
 \sup_{A\in\C}\|S(A)-A\|_\infty
 \leq\beta+M<\infty, \qquad S(A):=U^*T(A)U.
\]
$S$ is a standard $\varepsilon$-isometry, so Theorem \ref{thm:trace-rigidity}, with $D=A$ and $E=0$, gives
\[
 \|T(A)-UAU^*\|_1=\|S(A)-A\|_1\leq3\varepsilon,
 \qquad A\in\C.
\]
\end{proof}

\begin{unnumberedremark}
In the cone stability results, the approximating Wigner symmetry is unique as a map on $\C$. Indeed, suppose that two Wigner symmetries $\Phi_1,\Phi_2$ satisfy
\[
 \sup_{A\in\C}\mathsf d(T(A),\Phi_i(A))\leq M_i<\infty,
 \qquad i=1,2.
\]
For $A\in\C$ and $t>0$, homogeneity gives
\[
 t^\alpha\mathsf d(\Phi_1(A),\Phi_2(A))
 =\mathsf d(\Phi_1(tA),\Phi_2(tA))\leq M_1+M_2,
\]
where $\alpha=1$ for $\mathsf d=\done$ and $\alpha=\tfrac12$ for $\mathsf d=\db$. Letting $t\to\infty$ gives $\Phi_1=\Phi_2$. The implementing operator itself is not unique, since multiplying $U$ by a scalar of modulus one leaves $UAU^*$ unchanged.
\end{unnumberedremark}

\section{State spaces in finite and infinite dimensions}\label{sec:spheres} The cone arguments use large scalar multiples, which are unavailable on $\St$. We extend the logarithmic spiral of \cite{CW} from pure states to all states. The phases depend on the logarithms of the diagonal entries, so entries of very different sizes can acquire very different phases. In the fidelity comparison of two states with diagonals $(p_i)$ and $(q_i)$, these changes are weighted by $\sqrt{p_iq_i}$. The estimates below use these weights to keep distance distortion small, while growing dimension allows the map to remain a fixed distance from every unitary or antiunitary conjugation.

Suppose that $H$ is infinite dimensional, fix a countably infinite orthonormal sequence $(e_i)$, and let $E_0$ be its orthogonal complement. For $\rho\in\St$, let
\[
 p_i=\langle\rho e_i,e_i\rangle,
\]
and, for $\tau\in\mathbb R$, define the diagonal unitary
\begin{equation}
 D_\tau(\rho)e_k=
 \begin{cases}
 p_k^{\,i\tau/2}e_k,&p_k>0,\\
 e_k,&p_k=0
 \end{cases}
 \label{eq:state-phase}
\end{equation}
where $p_k^{\,i\tau/2}$ means $\exp((i\tau/2)\log p_k)$, which has absolute value one so $D_\tau(\rho)$ is unitary. We let $D_\tau(\rho)$ be the identity on $E_0$ and put
\begin{equation}
 \Theta_\tau(\rho)=D_\tau(\rho)\rho D_\tau(\rho)^*.
 \label{eq:twist-map}
\end{equation}
Conjugation by this diagonal unitary does not change the numbers $p_i$,
\[
 D_{-\tau}(\Theta_\tau(\rho))=D_\tau(\rho)^*,
\]
and $\Theta_{-\tau}$ is the inverse of $\Theta_\tau$. Hence the map is bijective. Unitary conjugation also preserves rank, so $\Theta_\tau$ maps pure states, which are the rank one states, to pure states.

We now check that $\Theta_\tau$ is a homeomorphism. Suppose that $\rho_n\to\rho$ in trace norm, and write $D_n=D_\tau(\rho_n)$ and $D=D_\tau(\rho)$. Since $\|\rho^{1/2}\|_2=1$, the triangle inequality and~\eqref{eq:schatten-holder} give
\[
 \|\Theta_\tau(\rho_n)-\Theta_\tau(\rho)\|_1
 \leq\|\rho_n-\rho\|_1+2\|(D_n-D)\rho^{1/2}\|_2.
\]
We put $D_ne_i=u_{n,i}e_i$, $De_i=u_ie_i$ and $p_i=\langle\rho e_i,e_i\rangle$. The operators $D_n$ and $D$ agree on $E_0$, so
\[
 \|(D_n-D)\rho^{1/2}\|_2^2
 =\sum_i p_i|u_{n,i}-u_i|^2\to0.
\]
Indeed, $u_{n,i}\to u_i$ whenever $p_i>0$, and dominated convergence applies with summable majorant $4p_i$. The same argument applies to the inverse $\Theta_{-\tau}$. By~\eqref{eq:state-metric-comparison}, the two metrics induce the same topology, so this proves the assertion for both. Each $\Theta_\tau$ also preserves spectra, since its value at every state is a unitary conjugate of that state.

Fix $\tau>0$. For two states $\rho,\sigma$, a common unitary conjugation turns their images into $\rho$ and $W\sigma W^*$, where $W=D_\tau(\rho)^*D_\tau(\sigma)$. By \eqref{eq:fidelity-trace-norm}, the change in fidelity is bounded by $\|\rho^{1/2}(W-I)\sigma^{1/2}\|_1$. This explains the quantity in the following lemma.

\begin{lemma}\label{lem:common-error} For every $\tau>0$ and all $\rho,\sigma\in\St$,
\begin{equation}
 \|\rho^{1/2}(D_\tau(\rho)^*D_\tau(\sigma) - I)\sigma^{1/2}\|_1\leq2\tau.
 \label{eq:Gamma-bound}
\end{equation}
\end{lemma}

\begin{proof} On the vector $e_i$, the operator $W:= D_\tau(\rho)^*D_\tau(\sigma)$ has diagonal entry
\[
 \exp\left(i\tau\log\sqrt{q_i/p_i}\right), \qquad q_i=\langle\sigma e_i,e_i\rangle
\]
when $p_iq_i>0$. Terms for which $p_iq_i=0$ make no contribution below. For the rank one operator occurring in the expansion of $\Gamma:=\rho^{1/2}(W-I)\sigma^{1/2}$, we have
\[
 \|\rho^{1/2}(e_i\otimes e_i)\sigma^{1/2}\|_1
 =\|\rho^{1/2}e_i\|\,\|\sigma^{1/2}e_i\|
 =\sqrt{p_iq_i}.
\]
By the triangle inequality one has
\begin{align*}
 \|\Gamma\|_1
 &\leq\sum_i\sqrt{p_iq_i}
 \left|\exp\left(i\tau\log\sqrt{q_i/p_i}\right)-1\right|\\
 &\leq\tau\sum_i\sqrt{p_iq_i}\left|\log\sqrt{q_i/p_i}\right|\\
 &\leq\tau\sum_i|q_i-p_i|\leq2\tau.
\end{align*}
In the second line we used $|e^{ix}-1|\leq|x|$. For $x>0$, $|\log x|\leq|x-x^{-1}|$: when $x\geq1$, this follows from $\log x\leq x-1\leq x-x^{-1}$, and the case $0<x<1$ follows after replacing $x$ by $x^{-1}$. We take $x=\sqrt{q_i/p_i}$ and obtain
\[
 \sqrt{p_iq_i}\left|\log\sqrt{q_i/p_i}\right|
 \leq|q_i-p_i|.
\]
Finally, $\sum_i p_i\leq\Tr\rho=1$ and $\sum_iq_i\leq\Tr\sigma=1$. The same estimates applied to the tails show that the finite partial sums converge to $\Gamma$ in trace norm, which justifies the expansion. \end{proof}

\begin{lemma}\label{lem:block-trace} If $A,B$ are states and $W$ is unitary, then
\begin{equation}
 \left|\|A-WBW^*\|_1-\|A-B\|_1\right|
 \leq4\sqrt{\|A^{1/2}(W-I)B^{1/2}\|_1}.
 \label{eq:block-trace-lemma}
\end{equation}
\end{lemma}

\begin{proof} The previous estimate controls $A^{1/2}(W-I)B^{1/2}$. We place this operator in the off diagonal blocks of the difference of two positive Gram operators, so that the square root estimate \eqref{eq:powers-stormer} applies. Let us regard $X=[A^{1/2},B^{1/2}]$ and $X'=[A^{1/2},WB^{1/2}]$ as operators from $H\oplus H$ to $H$. For
\[
 J:=\begin{bmatrix}I&0\\0&-I\end{bmatrix},
\]
put
\[
 \begin{aligned}
 G&=X^*X=\begin{bmatrix}
 A&A^{1/2}B^{1/2}\\
 B^{1/2}A^{1/2}&B
 \end{bmatrix},\\
 G'&=X'^*X'=\begin{bmatrix}
 A&A^{1/2}WB^{1/2}\\
 B^{1/2}W^*A^{1/2}&B
 \end{bmatrix}.
 \end{aligned}
\]
Then
\[
 XJX^*=A-B,\qquad X'JX'^*=A-WBW^*.
\]
If $X=VG^{1/2}$ is the polar decomposition, then $XJX^*=VG^{1/2}JG^{1/2}V^*$. The middle operator $G^{1/2}JG^{1/2}$ is supported on the initial space of the partial isometry $V$, so multiplication by $V$ and $V^*$ does not change its trace norm. The same argument applies to $X'$, so
\[
 \|A-B\|_1=\|G^{1/2}JG^{1/2}\|_1,
 \qquad
 \|A-WBW^*\|_1=\|G'^{1/2}JG'^{1/2}\|_1.
\]
The identity
\[
 RJR-R'JR'=(R-R')JR+R'J(R-R'), \qquad R:=  G^{1/2}, R' := (G')^{1/2}
\]
and \eqref{eq:schatten-holder} show that the difference of the two trace norms is at most
\[
 (\|R\|_2+\|R'\|_2)\|R-R'\|_2
 =2\sqrt2 \|G^{1/2}-G'^{1/2}\|_2
\]
where $\|R\|_2^2=\Tr G=\Tr A+\Tr B=2$, and similarly for $R'$.

The square root estimate \eqref{eq:powers-stormer}, applied to the positive operators $G,G'$ on $H\oplus H$, gives
\[
 \|G^{1/2}-G'^{1/2}\|_2^2
 \leq\|G-G'\|_1.
\]
The diagonal blocks of $G-G'$ are zero, and its off diagonal blocks are $-A^{1/2}(W-I)B^{1/2}$ and its adjoint. A self-adjoint block operator $\left[\begin{smallmatrix}0&Z\\Z^*&0\end{smallmatrix}\right]$ has each singular value of $Z$ twice, so
\[
 \|G-G'\|_1
 =2\|A^{1/2}(W-I)B^{1/2}\|_1.
\]
Combining the last three estimates gets us \eqref{eq:block-trace-lemma}. \end{proof}

\begin{lemma}\label{lem:twist-distortion} For every $\tau>0$ and all $\rho,\sigma\in\St$,
\begin{align}
 \big|\db(\Theta_\tau\rho,\Theta_\tau\sigma)
        -\db(\rho,\sigma)\big|
 &\leq2\sqrt\tau,
 \label{eq:bures-twist-distortion}\\
 \big|\|\Theta_\tau\rho-\Theta_\tau\sigma\|_1
        -\|\rho-\sigma\|_1\big|
 &\leq4\sqrt{2\tau}.
 \label{eq:trace-twist-distortion}
\end{align}
\end{lemma}

\begin{proof} Set $U=D_\tau(\rho)$ and $V=D_\tau(\sigma)$, so that (by definition of $\Theta_\tau$) $\Theta_\tau\rho=U\rho U^*$ and $\Theta_\tau\sigma=V\sigma V^*$. Both distances are invariant under applying the same unitary conjugation to both states. Applying $X\mapsto U^*XU$ gives
\[
 (U^*\Theta_\tau(\rho)U,\,
  U^*\Theta_\tau(\sigma)U)
 =(\rho,\,U^*V\sigma V^*U)
 =(\rho,W\sigma W^*),
 \qquad W=U^*V.
\]
Thus, for either distance, we need to measure the effect of conjugating $\sigma$ by $W$.

We begin with the Bures estimate. One can compute that $(W\sigma W^*)^{1/2} =W\sigma^{1/2}W^*$; hence,
\[
 \Fid(\Theta_\tau\rho,\Theta_\tau\sigma)
 =\Fid(\rho,W\sigma W^*)
 =\|\rho^{1/2}W\sigma^{1/2}W^*\|_1
 =\|\rho^{1/2}W\sigma^{1/2}\|_1,
\]
where the last equality uses unitary invariance of the trace norm. The reverse triangle inequality now gives
\[
 \bigl|\Fid(\Theta_\tau\rho,\Theta_\tau\sigma)
            -\Fid(\rho,\sigma)\bigr|
 \leq\|\rho^{1/2}(W-I)\sigma^{1/2}\|_1.
\]

For states, \eqref{eq:bures-definition} reads $\db(\rho,\sigma)^2=2-2\Fid(\rho,\sigma)$. Hence the absolute difference of the squares of the two Bures distances is at most $4\tau$ by Lemma \ref{lem:common-error}. Since $|\sqrt x-\sqrt y|\leq\sqrt{|x-y|}$ for $x,y\geq0$, this proves \eqref{eq:bures-twist-distortion}.

For the trace metric, apply Lemma~\ref{lem:block-trace} with $A=\rho$, $B=\sigma$, and the present unitary $W$. Lemma \ref{lem:common-error} bounds its right side by $4\sqrt{2\tau}$, which proves \eqref{eq:trace-twist-distortion}. \end{proof}

We next show that these maps do not approach any exact isometry. Let $z=\sum_i z_ie_i$ be a unit vector and let $P_z=z\otimes z$ be the corresponding pure state. Its diagonal entries are $p_i=|z_i|^2$, and hence $D_\tau(P_z)z$ has coordinates $z_i|z_i|^{i\tau}$. Thus, on pure states supported on the closed span of $(e_k)$, the map in \eqref{eq:twist-map} is induced by
\begin{equation} (z_k)_k\mapsto\bigl(z_k|z_k|^{i\tau}\bigr)_k. \label{eq:spiral} \end{equation}
For $d\geq3$, take a unit vector with half its squared norm in one coordinate and half equally distributed over the other $d-1$ coordinates. The two coordinate magnitudes have ratio $\sqrt{d-1}$, so their phases under \eqref{eq:spiral} differ by $\frac{\tau}{2}\log(d-1)$. Making this difference $\pi$ sends the vector to an orthogonal one: the two halves of its inner product with the image cancel. To obtain this orthogonality, we choose
\begin{equation}
 \tau_d=\frac{2\pi}{\log(d-1)}.
 \label{eq:tau-d}
\end{equation}
Moving a vector far from itself only rules out approximation by the identity. To exclude every conjugation, we use the following lower bound against complex linear mappings. The parameter used by \cite{CW} is $2\tau_d$, so our choice gives $(2\tau_d)\log(d-1)=4\pi$.

Let $M_d$ denote the complex vector space of $d\times d$ matrices. For $z=(z_1,\ldots,z_d)\in\mathbb C^d$, define
\[
 g_\tau(z):=\bigl(z_j|z_j|^{i\tau}\bigr)_{j=1}^d,
\]
where the value of a zero coordinate is zero.

\begin{lemma}\label{lem:spiral-linear-gap} Let $d\geq3$ and $\tau=2\pi/\log(d-1)$. For every complex linear map $L:M_d\to M_d$, there is a unit vector $z\in\mathbb C^d$ such that
\begin{equation}
 \bigl\|L(z\otimes z)-g_\tau(z)\otimes g_\tau(z)\bigr\|_\infty
 \geq\frac12.
 \label{eq:spiral-linear-gap}
\end{equation}
\end{lemma}

\begin{proof} We adapt the averaging argument in \cite[proof of Theorem~3(ii)]{CW}, using an additional test vector to improve the lower bound from $1/3$ to the optimal value $1/2$. We write $P_z=z\otimes z$ and put
\[
 \delta(L)=\sup_{\|z\|=1}
 \|L(P_z)-P_{g_\tau(z)}\|_\infty.
\]
Let $G$ be the group of matrices $D\Pi$, where $D$ is diagonal and unitary, and $\Pi$ is a permutation matrix. We average uniformly over the permutations and over the $d$ phases in $D$ and we let $\widetilde L$ denote the averaged map
\[
 \widetilde L(A) := \int_G U L(U^*AU)U^*\,dU
\]
whose role will be to simplify the arbitrary linear map $L$ while preserving the lower-bound argument. (Averaging will give $\widetilde L$ enough symmetry so that it has only three unknown coefficients.) For $U\in G$, we have $g_\tau(Uz)=Ug_\tau(z)$. By the triangle inequality: for every unit vector $z$,
\begin{align*}
 \|\widetilde L(P_z)-P_{g_\tau(z)}\|_\infty
 &\leq\int_G
 \|U L(P_{U^*z})U^*-U P_{g_\tau(U^*z)}U^*\|_\infty\,dU\\
 &\leq\delta(L),
\end{align*}
hence $\delta(\widetilde L)\leq\delta(L)$.

We now determine the form of the averaged mapping $\widetilde L$. The averaging gives
\[
 \widetilde L(UAU^*)=U\widetilde L(A)U^*,
 \qquad U\in G,\quad A\in M_d.
\]
We first use this identity for diagonal unitaries. Let $E_{ij}$ denote the matrix units and $D_\theta:=\operatorname{diag}(e^{i\theta_1},\ldots,e^{i\theta_d})$; then conjugation of the unit by $D_\theta$ gives
\[
 D_\theta E_{ij}D_\theta^*
 =e^{i(\theta_i-\theta_j)}E_{ij}.
\]
For $i\ne j$, these phase functions are distinct and nonconstant.

Since $\widetilde L$ commutes with every such conjugation, it cannot mix off-diagonal matrix units, and it sends diagonal matrices to diagonal matrices. Permutation symmetry then makes the multipliers on the off-diagonal entries equal. On diagonal inputs, it makes the $i$th output entry a linear combination of the $i$th input entry and the sum of all the entries.

Thus there are $\alpha,\beta,\gamma\in\mathbb C$ such that
\begin{equation}
 \widetilde L(A)=\alpha\Tr(A)I+\beta A+\gamma\operatorname{diag}(A),
 \qquad A\in M_d.
 \label{eq:averaged-linear-map}
\end{equation}
Consider the unit vector
\[
 \varphi:=\frac1{\sqrt2}
 \left(1,\frac1{\sqrt{d-1}},\ldots,
                 \frac1{\sqrt{d-1}}\right),
 \qquad w=g_\tau(\varphi).
\]
The phase of the first coordinate of $w$ differs from the common phase of the other coordinates by $\pi$, from which we deduce $\langle\varphi,w\rangle=0$. Let $R:=P_\varphi$, $S:=P_w$, and
\[
 c :=\sum_{j=1}^d|\varphi_j|^4=\frac{d}{4(d-1)},
\]
and observe that the matrices $R$ and $S$ have the same diagonal. Testing $\widetilde L(R)-S$ on both $w$ and $\varphi$ gives
\begin{align}
 \delta(\widetilde L)
 &\geq\max\left\{
  \left|\alpha+\gamma c-1\right|,
  \left|\alpha+\beta+\gamma c\right|\right\}\\
 &\geq\frac{|\beta+1|}{2}.
 \label{eq:first-beta-bound}
\end{align}

Consider the unit vector
\[
 \psi:=d^{-1/2}(1,\ldots,1).
\]
All its coordinates have the same modulus, so $P_{g_\tau(\psi)}=P_\psi$. By \eqref{eq:averaged-linear-map},
\[
 \widetilde L(P_\psi)-P_\psi
 =(\alpha+\gamma/d)I+(\beta-1)P_\psi.
\]
Its eigenvalues on $\mathbb C\psi$ and $\psi^\perp$ differ by $\beta-1$, so
\begin{equation}
 \delta(\widetilde L)\geq\frac{|\beta-1|}{2}.
 \label{eq:second-beta-bound}
\end{equation}
Together with \eqref{eq:first-beta-bound}, this gives
\[
 \delta(L)\geq\delta(\widetilde L)
 \geq\frac12\max\{|\beta+1|,|\beta-1|\}
 \geq\frac12.
\]
The coefficient $1/2$ is optimal: the map $L(A)=\frac12\Tr(A)I$ has error $1/2$ for every rank one projection.

The function $g_\tau$ is continuous at zero coordinates as well, since the absolute value of $z_j|z_j|^{i\tau}$ is $|z_j|$. The unit sphere of $\mathbb C^d$ is compact, so the supremum in the definition of $\delta(L)$ is attained. \end{proof}

\begin{proof}[Proof of Theorem \ref{thm:sphere-instability}] We choose strictly increasing integers $d_n\geq3$ with $d_n\to\infty$, and put
\[
 \Theta_n=\Theta_{\tau_{d_n}},\qquad
 \varepsilon_n=8\sqrt{\frac{\pi}{\log(d_n-1)}}.
\]
For each $n$, we write $d=d_n$. Lemma \ref{lem:twist-distortion} gives the two errors
\begin{equation}
 \sqrt{\frac{8\pi}{\log(d-1)}}
 \quad\hbox{and}\quad
 8\sqrt{\frac{\pi}{\log(d-1)}}
 \label{eq:distortion-rates}
\end{equation}
for $\db$ and $\done$, respectively. These are obtained by substituting $\tau_d=2\pi/\log(d-1)$ into $2\sqrt\tau$ and $4\sqrt{2\tau}$. Both tend to zero, and the second one is $\varepsilon_n$, showing that each $\Theta_n$ is an $\varepsilon_n$-isometry for both metrics.

Let $U$ be a unitary or antiunitary on $H$ and $E_d:=\operatorname{span}\{e_1,\ldots,e_d\}$, with orthogonal projection $P_d$. If $U$ is unitary, set $L(A)=P_dUAU^*P_d$; if $U$ is antiunitary, let $K$ be complex conjugation in an orthonormal basis extending $(e_1,\ldots,e_d)$, and set $L(A)=P_d(UK)A^{\mathsf T}(UK)^*P_d$. In either case $L:M_d\to M_d$ is complex linear and $L(A)=P_dUAU^*P_d$ for every self-adjoint $A\in M_d$.

Applying Lemma \ref{lem:spiral-linear-gap} to the appropriate compressed map, we obtain a rank one projection $P$ supported on $E_d$ such that
\[
 \|P_d(\Theta_{\tau_d}(P)-UPU^*)P_d\|_\infty\geq\frac12;
\]
the same lower bound holds without $P_d$ since compressions do not increase the operator norm. We recall the two elementary formulas which convert this operator norm gap into the two metric gaps. If $P=x\otimes x$ and $Q=y\otimes y$ are rank one projections, where $x,y$ are unit vectors, then $\Fid(P,Q)=|\langle x,y\rangle|$. The possibly nonzero eigenvalues of $P-Q$ are $\pm\sqrt{1-|\langle x,y\rangle|^2}$. As a result,
\[
 \|P-Q\|_1=2\|P-Q\|_\infty,
 \qquad
 \db(P,Q)^2=2-2\sqrt{1-\|P-Q\|_\infty^2}.
\]
We apply these identities to $\Theta_{\tau_d}(P)$ and $UPU^*$. By the operator norm gap $1/2$, we obtain
\[
 \|\Theta_{\tau_d}(P)-UPU^*\|_1\geq1,
 \qquad
 \db(\Theta_{\tau_d}(P),UPU^*)
 \geq\sqrt{2-2\sqrt{1-\frac14}}
 =\sqrt{2-\sqrt3}.
\]
For every $U$ there is such a projection $P$, which proves \eqref{eq:fixed-bures-gap} and \eqref{eq:fixed-trace-gap}.

This proof has just shown instability on a fixed infinite-dimensional Hilbert space $H$, but this construction can be run afresh on $\mathbb C^d$ using its standard basis. The resulting maps have distortion $\varepsilon_d\to0$ for both metrics and their distance from every unitary or antiunitary conjugation remains at least $1$ and $\sqrt{2-\sqrt3}$ in trace and Bures distance respectively, which implies that no stability estimate tending to 0 with the distortion can hold uniformly in $d$.

\end{proof}

\begin{remark} The map in \eqref{eq:twist-map} extends to the cone by
\[
 \widehat\Theta_\tau(A)=
 \begin{cases}
  (\Tr A)\Theta_\tau(A/\Tr A),&A\neq0,\\
  0,&A=0.
 \end{cases}
\]
If $A$ and $B$ are replaced by $tA$ and $tB$, every trace distance error is multiplied by $t$, while every Bures distance error is multiplied by $\sqrt t$, so this extension cannot have a bounded additive error on the cone unless every error is zero. \end{remark}

The preceding examples rule out estimates which are independent of the dimension; however, in each fixed finite dimension, compactness gives a stability estimate.

\begin{proof}[Proof of Corollary \ref{cor:fixed-dimensional}] We will show that a sequence of maps whose distortions tend to zero has a subsequence converging uniformly to an exact surjective isometry. Compactness gives convergent images on a dense set, and the distance inequality upgrades this convergence to uniform convergence.

If the corollary failed, there would be a $\delta>0$, numbers $\varepsilon_n\to0$, and $\varepsilon_n$-isometries $T_n$ whose uniform distance from every map of the form $A\mapsto UAU^*$ is at least $\delta$. The state space is compact in finite dimension: it is closed and bounded in the finite dimensional space of self-adjoint matrices. The two metrics give the same topology by \eqref{eq:state-metric-comparison}.

We let $(x_k)$ be a countable dense subset of this compact state space. By compactness there is a subsequence for which $T_n(x_1)$ converges; we choose a further subsequence for which $T_n(x_2)$ also converges, and continue. The diagonal subsequence makes $T_n(x_k)$ convergent for every $k$. We define $T(x_k)=\lim_nT_n(x_k)$. For $j,k$,
\[
 \mathsf d(T(x_j),T(x_k))
 =\lim_n\mathsf d(T_n(x_j),T_n(x_k))
 =\mathsf d(x_j,x_k),
\]
because $\varepsilon_n\to0$. Thus $T$ is an isometry on the dense set; by taking limits, it extends uniquely to an isometry of the whole state space.

The convergence is uniform even though the maps $T_n$ need not be continuous. Indeed, we choose a finite $\eta$-net from the points $x_k$. For $x$ and a net point $z$ with $\mathsf d(x,z)<\eta$, the triangle inequality, the $\varepsilon_n$-isometry inequality, and the fact that $T$ is an isometry give
\[
 \mathsf d(T_nx,Tx)
 \leq 2\eta+\varepsilon_n+\mathsf d(T_nz,Tz).
\]
For fixed $\eta$, the last term tends to zero at every point of the finite net, and hence uniformly over that net. Taking the upper limit in $n$ and then letting $\eta\to0^+$ proves uniform convergence.

The limiting isometry $T$ is necessarily surjective. To see this, write $K$ for the compact state space. If $y\notin T(K)$, then $a:=\operatorname{dist}_{\mathsf d}(y,T(K))>0$, since $T(K)$ is compact. For integers $n>m\geq0$, the isometry property gives
\[
 \mathsf d(T^ny,T^my)
 =\mathsf d(T^{n-m}y,y)\geq a,
\]
because $T^{n-m}y\in T(K)$. This contradicts compactness of $K$, so $T$ is surjective. Theorem \ref{thm:MT} gives $T(A)=UAU^*$ for a unitary or antiunitary $U$, which contradicts the fixed uniform gap. \end{proof}

\begin{remark} Compactness gives a stability modulus but no quantitative rate. In particular, it does not establish a bound of the form $C_d\varepsilon$. Equation \eqref{eq:distortion-rates} shows that any such bound would require $C_d$ to grow at least as fast as $\sqrt{\log d}$ for both metrics. \end{remark} \section*{Acknowledgements}

I am grateful to Denny Leung for suggesting these stability questions, for several helpful discussions, and for pointing out the distinction between the positive cone and the state space. I also thank him for his encouragement and for drawing my attention to the broader context of isometries of positive trace-class operators, and for feedback and comments on an earlier version of this manuscript. AI assistance was used for literature search.


\begin{thebibliography}{99}

\bibitem{BJL} R.~Bhatia, T.~Jain and Y.~Lim, \emph{On the Bures-Wasserstein distance between positive definite matrices}, Expo. Math. \textbf{37} (2019), 165-191. \href{https://doi.org/10.1016/j.exmath.2018.01.002}{doi:10.1016/j.exmath.2018.01.002}.

\bibitem{BBI} D.~Burago, Y.~Burago and S.~Ivanov, \emph{A Course in Metric Geometry}, Graduate Studies in Mathematics, vol.~33, American Mathematical Society, Providence, RI, 2001. \href{https://doi.org/10.1090/gsm/033}{doi:10.1090/gsm/033}.

\bibitem{Busch} P.~Busch, \emph{Stochastic isometries in quantum mechanics}, Math. Phys. Anal. Geom. \textbf{2} (1999), 83-106. \href{https://doi.org/10.1023/A:1009822315406}{doi:10.1023/A:1009822315406}.

\bibitem{CD} L.~Cheng and Y.~Dong, \emph{A note on the stability of nonsurjective $\varepsilon$-isometries of Banach spaces}, Proc. Amer. Math. Soc. \textbf{148} (2020), 4837-4844. \href{https://doi.org/10.1090/proc/15110}{doi:10.1090/proc/15110}.

\bibitem{CDZ} L.~Cheng, Y.~Dong and W.~Zhang, \emph{On stability of nonlinear non-surjective $\varepsilon$-isometries of Banach spaces}, J. Funct. Anal. \textbf{264} (2013), 713-734. \href{https://doi.org/10.1016/j.jfa.2012.11.008}{doi:10.1016/j.jfa.2012.11.008}.

\bibitem{CTZ} L.~Cheng, K.~Tu and W.~Zhang, \emph{On weak stability of $\varepsilon$-isometries on wedges and its applications}, J. Math. Anal. Appl. \textbf{433} (2016), 1673-1689. \href{https://doi.org/10.1016/j.jmaa.2015.08.033}{doi:10.1016/j.jmaa.2015.08.033}.

\bibitem{CW} J.~Cuesta and M.~M.~Wolf, \emph{Are almost-symmetries almost linear?}, J. Math. Phys. \textbf{60} (2019), 082101. \href{https://doi.org/10.1063/1.5087539}{doi:10.1063/1.5087539}.

\bibitem{DLL} Y.~Dong, D.~H.~Leung and L.~Li, \emph{Stability of isometries between the positive cones of ordered Banach spaces}, Math. Ann. \textbf{389} (2024), 253-280. \href{https://doi.org/10.1007/s00208-023-02649-z}{doi:10.1007/s00208-023-02649-z}.

\bibitem{DLLZ} Y.~Dong, D.~H.~Leung, L.~Li and B.~Zhang, \emph{Stability of isometries between the positive cones of $L^1$ and trace-class operators}, Israel J. Math., to appear.

\bibitem{FvdG} C.~A.~Fuchs and J.~van de Graaf, \emph{Cryptographic distinguishability measures for quantum-mechanical states}, IEEE Trans. Inform. Theory \textbf{45} (1999), 1216-1227. \href{https://doi.org/10.1109/18.761271}{doi:10.1109/18.761271}.

\bibitem{Gevirtz} J.~Gevirtz, \emph{Stability of isometries on Banach spaces}, Proc. Amer. Math. Soc. \textbf{89} (1983), 633-636. \href{https://doi.org/10.1090/S0002-9939-1983-0718987-6}{doi:10.1090/S0002-9939-1983-0718987-6}.

\bibitem{Had} N.~Hadjisavvas, \emph{Metrics on the set of states of a $W^*$-algebra}, Linear Algebra Appl. \textbf{84} (1986), 281-287. \href{https://doi.org/10.1016/0024-3795(86)90320-4}{doi:10.1016/0024-3795(86)90320-4}.

\bibitem{HU} D.~H.~Hyers and S.~M.~Ulam, \emph{On approximate isometries}, Bull. Amer. Math. Soc. \textbf{51} (1945), 288-292. \href{https://doi.org/10.1090/S0002-9904-1945-08337-2}{doi:10.1090/S0002-9904-1945-08337-2}.

\bibitem{MZC} Z.~Ma, F.-L.~Zhang and J.-L.~Chen, \emph{Fidelity induced distance measures for quantum states}, Phys. Lett. A \textbf{373} (2009), 3407-3409. \href{https://doi.org/10.1016/j.physleta.2009.07.042}{doi:10.1016/j.physleta.2009.07.042}.

\bibitem{MT} L.~Moln\'ar and W.~Timmermann, \emph{Isometries of quantum states}, J. Phys. A: Math. Gen. \textbf{36} (2003), 267-273. \href{https://doi.org/10.1088/0305-4470/36/1/318}{doi:10.1088/0305-4470/36/1/318}.

\bibitem{OS} M.~Omladi\v{c} and P.~\v{S}emrl, \emph{On non linear perturbations of isometries}, Math. Ann. \textbf{303} (1995), 617-628. \href{https://doi.org/10.1007/BF01461008}{doi:10.1007/BF01461008}.

\bibitem{PS} R.~T.~Powers and E.~St{\o}rmer, \emph{Free states of the canonical anticommutation relations}, Comm. Math. Phys. \textbf{16} (1970), 1-33.

\bibitem{SV} P.~\v{S}emrl and J.~V\"ais\"al\"a, \emph{Nonsurjective nearisometries of Banach spaces}, J. Funct. Anal. \textbf{198} (2003), 268-278. \href{https://doi.org/10.1016/S0022-1236(02)00049-6}{doi:10.1016/S0022-1236(02)00049-6}.

\bibitem{Sun} L.~Sun, \emph{Hyers-Ulam stability of $\varepsilon$-isometries between the positive cones of $L^p$-spaces}, J. Math. Anal. Appl. \textbf{487} (2020), 124014. \href{https://doi.org/10.1016/j.jmaa.2020.124014}{doi:10.1016/j.jmaa.2020.124014}.

\bibitem{SSW} L.~Sun, Y.~Sun and S.~Wang, \emph{On perturbed isometries between the positive cones of certain continuous function spaces}, Results Math. \textbf{78} (2023), article 63. \href{https://doi.org/10.1007/s00025-023-01844-3}{doi:10.1007/s00025-023-01844-3}.

\bibitem{Uhlmann} A.~Uhlmann, \emph{The ``transition probability'' in the state space of a $^*$-algebra}, Rep. Math. Phys. \textbf{9} (1976), 273-279.

\bibitem{Vestfrid15} I.~A.~Vestfrid, \emph{Stability of almost surjective $\varepsilon$-isometries of Banach spaces}, J. Funct. Anal. \textbf{269} (2015), 2165-2170. \href{https://doi.org/10.1016/j.jfa.2015.04.009}{doi:10.1016/j.jfa.2015.04.009}.

\bibitem{Vestfrid23} I.~A.~Vestfrid, \emph{$\varepsilon$-isometries between the positive cones of continuous functions spaces}, Israel J. Math. \textbf{253} (2023), 989-1000. \href{https://doi.org/10.1007/s11856-022-2393-4}{doi:10.1007/s11856-022-2393-4}.

\end{thebibliography}
 \end{document}